\documentclass[11pt]{amsart}
\usepackage[english]{babel}
\usepackage{amsmath}
\usepackage{amsfonts}
\usepackage{amssymb}
\usepackage{amsthm}
\usepackage{tikz}
\usetikzlibrary{calc}
\usepgflibrary{fpu}

\usepackage[english]{babel}
\usepackage{yfonts}
\usepackage[T1]{fontenc}
\usepackage[utf8x]{inputenc}
\usepackage{enumerate}
\usepackage{verbatim}
\usepackage{graphicx}
\usepackage{wrapfig}
\usepackage{verbatim}
\usepackage{faktor}
\usepackage{xcolor}
\usepackage[all]{xy}

\newcommand{\calF}{\mathcal F}
\newcommand{\E}{{\rm E}}
\newcommand{\Sp}{{\rm Sp}}
\newcommand{\Hom}{{\rm Hom}}

\newcommand{\rk}{{\rm rk}}
\newcommand{\wt}{\widetilde}

\newcommand{\deH}{\partial\mathbb H _{\mathbb C}}

\newcommand{\<}{\langle}
\renewcommand{\>}{\rangle}

\newcommand{\stab}{\mathrm{stab}}
\newcommand{\Stab}{\mathrm{Stab}}

\newcommand{\GL}{{\rm GL}}
\newcommand{\SU}{{\rm SU}}
\newcommand{\U}{{\rm U}}
\newcommand{\SL}{{\rm SL}}
\newcommand{\SO}{{\rm SO}}
\newcommand{\PSL}{{\rm PSL}}
\newcommand{\PU}{{\rm PU}}

\newcommand{\Gr}{{\rm Gr}}

\newcommand{\Id}{{\rm Id}}
\newcommand{\ov}{\overline}

\renewcommand{\b}{\beta}
\newcommand{\g}{\gamma}

\renewcommand{\k}{\kappa}

\renewcommand{\o}{\omega}
\newcommand{\G}{\Gamma}

\newcommand{\Om}{\Omega}

\newcommand{\C}{\mathbb C}
\newcommand{\D}{\mathbb D}
\renewcommand{\H}{\mathbb H}

\newcommand{\R}{\mathbb R}

\newcommand{\Aa}{\mathcal A}
\newcommand{\Bb}{\mathcal B}
\newcommand{\Cc}{\mathcal C}
\newcommand{\Dd}{\mathcal D}

\newcommand{\Ff}{\mathcal F}

\newcommand{\Hh}{\mathcal H}

\newcommand{\Ll}{\mathcal L}

\newcommand{\Oo}{\mathcal O}

\newcommand{\Ss}{\mathcal S}
\newcommand{\Tt}{\mathcal T}
\newcommand{\Uu}{\mathcal U}

\newcommand{\Ww}{\mathcal W}
\newcommand{\Xx}{\mathcal X}
\newcommand{\Yy}{\mathcal Y}

\newcommand{\fru}{\mathfrak{u}}

\newcommand{\tr}{\text{tr}}
\newcommand{\tra}{\pitchfork}
\newcommand{\Isom}{\text{Isom}}
\usepackage{stmaryrd}
\numberwithin{equation}{section}

\numberwithin{equation}{section}

\newcommand{\bqn}{\begin{equation*}}
\newcommand{\eqn}{\end{equation*}}

\newcommand{\bq}{\begin{equation}}
\newcommand{\eq}{\end{equation}}
\newcommand{\ba}{\begin{aligned}}
\newcommand{\ea}{\end{aligned}}
\newcommand{\be}{\begin{enumerate}}
\newcommand{\ee}{\end{enumerate}}
\newcommand{\bsm}{\left[\begin{smallmatrix}}
\newcommand{\esm}{\end{smallmatrix}\right]}                   
\newcommand{\bpm}{\begin{bmatrix}}
\newcommand{\epm}{\end{bmatrix}}
\newcommand{\barr}{\begin{displaymath}\begin{array}{cccc}}
\newcommand{\earr}{\end{array}\end{displaymath}}
\newcommand{\barrl}{\begin{displaymath}\begin{array}{lcl}}
\newcommand{\earrl}{\end{array}\end{displaymath}}
\newcommand{\barl}{\begin{displaymath}\begin{array}{l}}
\newcommand{\earl}{\end{array}\end{displaymath}}
\newcommand{\bxym}{ \begin{displaymath}\xymatrix }
\newcommand{\exym}{\end{displaymath}}

\theoremstyle{plain}
\newtheorem{thm}{Theorem}[section]
\newtheorem{lem}[thm]{Lemma}
\newtheorem{prop}[thm]{Proposition}
\newtheorem{cor}[thm]{Corollary}
\newtheorem*{teo*}{Theorem}

\newtheorem{claim}[thm]{Claim}

\theoremstyle{definition}
\newtheorem{defn}[thm]{Definition}

\newcommand{\thismonth}{\ifcase\month 
  \or January\or February\or March\or April\or May\or June%
  \or July\or August\or September\or October\or November%
  \or December\fi}

\begin{document}
\title[Maximal representations into $\SU(m,n)$]{Maximal representations of complex hyperbolic lattices into $\SU(m,n)$}

\author{M. B. Pozzetti}
\address{Department Mathematik, ETH Z\"urich, 
R\"amistrasse 101, CH-8092 Z\"urich, Switzerland}
\email{beatrice.pozzetti@math.ethz.ch}

\date{\today}
\thanks{I want to thank Marc Burger and Alessandra Iozzi for suggesting the topic of this article, for sharing with me their ideas and for many insightful conversation about the content of the paper. This work was partially supported by the Swiss National Science Foundation project 200020-144373. I am grateful for the hospitality of Princeton University and Universit\`a di Pisa where part of this work was completed.}

\begin{abstract}
Let $\G$ denote a lattice in $\SU(1,p)$, with $p$ greater than 1. We show that there exists no Zariski dense maximal representation with target $\SU(m,n)$ if $n>m>1$. The proof is geometric and is based on the study of the rigidity properties of the geometry whose points are isotropic $m$-subspaces of a complex vector space $V$ endowed with a Hermitian metric $h$ of signature $(m,n)$ and whose lines correspond to the $2m$ dimensional subspaces of $V$ on which the restriction of $h$ has signature $(m,m)$.  
\end{abstract}
\maketitle

\section{Introduction}\label{sec:Intro}
Let $\G$ be a finitely generated group and $G$ be a connected semisimple Lie group. It is an interesting problem to select and study some connected components of the representation variety $\Hom(\G,G)$ that consist of homomorphisms $\rho:\G\to G$ that are well behaved and, ideally, reflect some interesting geometric properties of the group $\G$. The best example of this framework is the case in which $\G$ is the fundamental group $\G_g$ of a closed surface of genus $g\geq 2$ and $G$ is $\PSL_2(\R)$. In this case the Teichm\"uller space arises as a component of $\Hom(\G_g,\PSL_2(\R))/\!/\PSL_2(\R)$ that can be selected by means of a cohomological invariant \cite{Goldmanthesis}.

In the more general setting in which $G$ is any Hermitian Lie group, the so-called \emph{maximal representations} form a well studied union of connected components  of the character variety  $\Hom(\G_g,G)/\!/G$ which generalize the Teichm\"uller component \cite{BGG,Toledo}. In analogy with holonomy representations of hyperbolizations, maximal representations can be characterized as those 
representations that maximize an invariant, the \emph{Toledo invariant}, that can be defined in terms of bounded cohomology. Such representations are discrete and faithful, and give rise to interesting geometric structures \cite{Toledo, Anosov}. Maximal representations in Hermitian Lie groups were first studied by Toledo in  \cite{Toledo-rep} where he proves that a maximal representation $\rho:\G_g\to \SU(1,q)$ fixes a complex geodesic, and by Hernandez \cite{Hernandez} who studied maximal representations $\rho:\G_g\to \SU(2,q)$ and showed that the image must stabilize a symmetric space associated to the group $\SU(2,2)$. In general any maximal representation stabilizes a tube-type subdomain \cite{Toledo}.
Despite this, a remarkable flexibility result holds for maximal representations $\rho$ of fundamental groups of surfaces: if the image of $\rho$ is a Hermitian Lie group of tube type, then $\rho$ admits a one parameter family of deformations consisting of Zariski dense representations \cite{Toledo, KP}.

An analogue of the Toledo invariant was defined by Burger and Iozzi in \cite{BIpreprint} for representations of a lattice $\G$ in $\SU(1,p)$ with values in a Hermitian Lie group $G$. This allows to select a union of connected components of $\Hom(\G,G)$ consisting of \emph{maximal representations}. These generalize maximal representations of fundamental groups of surfaces: the fundamental group of a surface is a lattice in $\PU(1,1)=\PSL_2(\R)$. However, if $p$ is greater than one, a different behavior is expected: Goldman and Milson  proved local rigidity for the standard embedding of $\G$ in $\SU(1,q)$ \cite{GM}, and Corlette 
proved that maximal representations of uniform complex hyperbolic lattices with values in $\SU(1,q)$ all come from the standard construction \cite{Corlette}. 
The picture for rank one targets was completed independently by Koziarz and Maubon \cite{KMrank1} and by Burger and Iozzi \cite{BICartan}: any maximal representation of a lattice in $\SU(1,p)$ with values in $\SU(1,q)$ admits an equivariant totally geodesic holomorphic embedding $\H^p_\C\to \H^q_\C$. 
Koziarz and Maubon generalized this result to the situation in which the target group is classical of rank 2 and the lattice is cocompact \cite{KM}\footnote{In his recent preprint Spinaci studies maximal representations of cocompact K\"ahler groups admitting an holomorphic equivariant map \cite{Spi}}. It is conjectured that every maximal representation of a complex hyperbolic lattice with target a Hermitian Lie group is superrigid, namely it extends, up to a representation of $\G$ in the compact centralizer of the image, to a representation of the ambient group $\SU(1,p)$. 

In this article we show that the conjecture holds for Zariski dense representations in $\SU(m,n)$, with $m$ different from $n$:
\begin{thm}\label{thm:Zariskisuperrigidity}
 Let $\G$ be a lattice in $\SU(1,p)$ with $p>1$. If $m$ is different from $n$, then every Zariski dense maximal representation of $\G$ into $\PU(m,n)$ is the restriction of a representation of $\SU(1,p)$.
\end{thm}

This immediately implies the following:
\begin{cor}\label{cor:noZariskidense} 
 Let $\G$ be a lattice in $\SU(1,p)$ with $p>1$. There are no Zariski dense maximal representations of $\G$ into $\SU(m,n)$, if $1<m<n$.
\end{cor}

Exploiting results of \cite{tight}, and the classification of maximal representations between Hermitian Lie groups \cite{Ham1, Ham2, HamP}, we are able to use our main theorem to give a structure theorem for all maximal representations $\rho:\G\to\SU(m,n)$.
\begin{thm}\label{thm:general}
 Let $\rho:\G\to \SU(m,n)$ be a maximal representation. Then the Zariski closure $L=\ov {\rho(\G)}^Z$ splits as the product $\SU(1,p)\times L_t\times K$ where $L_t$ is a Hermitian Lie group of tube type without irreducible factors that are virtually isomorphic to $\SU(1,1)$, and $K$ is a compact subgroup of $\SU(m,n)$.
 
Moreover there exists an integer $k$ such that the inclusion of $L$ in $\SU(m,n)$ can be realized as
$$\Delta\times i\times {\rm Id}:L
\to \SU(1,p)^{m-k}\times\SU(k,k)\times K<\SU(m,n)$$ 
where $\Delta:\SU(1,p)\to \SU(1,p)^{m-k}$ is the diagonal embedding, $i: L_t\to \SU(k,k)$ is a tight holomorphic embedding and $K$ is contained in the compact centralizer of $\Delta\times i(L)$. 
\end{thm}
It is possible to show that there are no tube-type factors in the Zariski closure of the image of $\rho$ by imposing some non-degeneracy hypothesis on the associated linear representation of $\G$ into $\GL(\C^{m+n})$:
\begin{cor}\label{cor:general}
Let $\G$ be a  lattice in $\SU(1,p)$, with $p>1$ and let $ \rho$ be a maximal representation of $\G$ into $\SU(m,n)$. Assume that the associated linear representation of $\G$ on $\C^{n+m}$ has no invariant subspace on which the restriction of the Hermitian form has signature $(k,k)$ for some $k$. Then 
\begin{enumerate}
 \item $n\geq pm$,
 \item $\rho$ is conjugate to $\overline \rho\times \chi_{\rho}$ where $\ov\rho$ is the restriction to $\G$ of the diagonal embedding of $m$ copies of $\SU(1,p)$ in $\SU(m,n)$ and  $\chi_\rho$ is a representation $\chi_{\rho}:\G\to K$, where $K$ is a compact group. 
\end{enumerate}
\end{cor}

Recently Klingler proved that all the representations of uniform complex hyperbolic lattices that satisfy a technical algebraic condition are locally rigid \cite{Klingler}. As a particular case his main theorem implies that if $\G$ is a cocompact lattice in $\SU(1,p)$ and $\rho:\G\to \SU(m,n)$ is obtained by restricting to $\G$ the diagonal inclusion of $\SU(1,p)$ in $\SU(m,n)$, then $\rho$ is locally rigid. Since the invariant defining the maximality of a representation is constant on connected components of the representation variety, we get a new proof of Klingler's result  in our specific case, and the generalization of this latter result to non-uniform lattices:
\begin{cor}\label{cor:local rigidity}
 Let $\G$ be a lattice in $\SU(1,p)$, with $p>1$, and let $\rho$ be the restriction to $\G$ of the diagonal embedding of $m$ copies of $\SU(1,p)$ in $\SU(m,n)$. Then $\rho$ is locally rigid.
\end{cor}

Our proof of Theorem \ref{thm:Zariskisuperrigidity} is inspired by Margulis' beautiful proof of superrigidity for higher rank lattices: in order to show that a representation $\rho:\G\to G$ extends to the group $H$ in which $\G$ sits as a lattice, it is enough to exhibit a $\rho$-equivariant algebraic map $\phi:H/P\to G/L$ for some parabolic subgroups $P$ of $H$ and $L$ of $G$. The existence of measurable $\rho$-equivariant boundary maps $\phi: H/P\to G/L$ where $P<H$ is a minimal parabolic subgroup and $G$ is a linear algebraic group is by now well understood \cite{SUpq, Fur, BF}, and the crucial part in the proof of superrigidity for our representations is to show that such a measurable equivariant boundary map must indeed be algebraic. 
In general not every representation of a complex hyperbolic lattice is superrigid: for example Livne constructed in his PhD dissertation a lattice in $\SU(1,2)$ that surjects onto a free group (cfr. \cite[Chapter 16]{DelMos}), moreover Mostow constructed examples of lattices $\G_1,\G_2$ in $\SU(1,2)$ admitting a surjection $\G_1\twoheadrightarrow \G_2$ with infinite kernel (cfr. \cite{Most, Tolmaps}).
These examples show that many representations of complex hyperbolic lattice do not extend to $\SU(1,2)$ and hence some additional information on the boundary map $\phi$ is needed in order to deduce its algebraicity.

We restrict our interest to maximal representations precisely to be able to gather some information on a measurable boundary map $\phi$. The maximality of a representation $\rho$ can be rephrased as a property of the induced pullback map $\rho^*:{\rm H}^2_{\rm cb}(G,\R)\to {\rm H}^2_{\rm b}(\G,\R) $ in bounded cohomology. One of the advantages of bounded cohomology with respect to ordinary cohomology is that it can be isometrically computed from the complex of $\rm L^\infty$ functions on some suitable boundary of the group \cite{BMGAFA} and, in all geometric cases known so far \cite{BIAppendix}, the pullback map in bounded cohomology can be implemented using boundary maps. In particular we exploit results of \cite{Formula} and we show that the fact that the representation $\rho$ is maximal implies that a $\rho$-equivariant measurable boundary map must preserve some incidence structure on the boundary (this was proven in \cite{BICartan} in the case in which the image is of rank one).

To describe more precisely this incidence structure, recall that one of the key features of the complex hyperbolic space is the existence of complex geodesics tangent to any vector in $T^1\H_\C^p$: these are precisely the totally geodesic holomorphic embeddings of the Poincar\'e disc in $\H_\C^p$. The boundaries of these subspaces produce a family of circles in $\deH^p$, the socalled \emph{chains}, that form an incidence structure  that was first studied by  Cartan in \cite{Cartan}.
Under many respects, the natural generalization to higher rank of the visual boundary of the complex hyperbolic space is the Shilov boundary of a Hermitian symmetric space and the generalization of a complex geodesic, when maximal representations are involved, is a  maximal tube-type subdomain. 
All these objects have an explicit linear description: it is well known that the boundary of the complex hyperbolic space can be identified with the set of isotropic lines in $\C^{p+1}$, and it is easy to check that a triple of lines $x,y,z$ is contained in a chain if and only if $\dim\,\<x,y,z\>=2$.
Similarly the Shilov boundary $\Ss_{m,n}$ of $\SU(m,n)$ can be described as the set of maximal isotropic subspaces of $\C^{m+n}$ and, again, a triple of transverse isotropic subspaces $x,y,z$ in $\Ss_{m,n}$ is contained in the boundary of a tube-type subdomain precisely when $\dim\,\<x,y,z\>=2m$. In such case we will say that $x,y,z$ are contained in an \emph{$m$-chain}.

As it turns out, 
if $\rho:\G\to \SU(m,n)$ is a maximal representation and $\phi:\deH^p\to \Ss_{m,n}$ is a measurable $\rho$-equivariant boundary map, 
 then $\phi$ induces a map from the chain geometry of $\deH^p$ to the geometry whose space is $\Ss_{m,n}$ and whose lines are the $m$-chains. 
Therefore most of this paper is devoted to  the study of these geometries. We generalize some results of Cartan \cite{Cartan} and Goldman \cite{Goldman} and this allows us to prove a strong rigidity result for measurable maps that preserve this geometry, that is a higher rank analogue of the main theorem of \cite{Cartan}:
\begin{thm}\label{thm:phirational}
 Let $p>1$, $1<m<n$ and let $\phi:\deH^p\to\Ss_{m,n}$ be a measurable map whose essential image is Zariski dense. Assume that
 , for almost every triple with $\dim\<x,y,z\>=2$, it holds $\dim\<\phi(x),\phi(y),\phi(z)\>=2m$. Then  $\phi$ coincides almost everywhere with a rational map.
\end{thm}
\subsection*{Outline of the paper}
In Section \ref{sec:Toledo}, after recalling the relevant concepts about Hermitian symmetric spaces and continuous bounded cohomology, we prove that a measurable boundary map associated with a maximal representation induces a map between chain geometries. Sections \ref{sec:S_mn} to \ref{sec:reduction} are devoted to prove Theorem \ref{thm:phirational}: in Section \ref{sec:S_mn} we study the chain geometry of $\Ss_{m,n}$ and prove some properties of the incidence structure of chains; in Section \ref{sec:chain} we show that the restriction to almost every chain of a measurable boundary map associated with a maximal representation is rational; in Section \ref{sec:reduction} we show that this information is already enough to conclude. We finish the article with  Section \ref{sec:rational}, where we prove all the remaining results announced in this introduction.
\section{Preliminaries}\label{sec:Toledo}

\subsection{Hermitian Symmetric spaces}
 Let $G$ be a connected semisimple Lie group of noncompact type with finite center and let $K$ be a maximal compact subgroup. We will denote by $\Xx=G/K$ the associated symmetric space. Throughout this article we will be only interested in \emph{Hermitian} symmetric spaces, that is in those symmetric spaces that admit a $G$-invariant complex structure $J$. It is a classical fact \cite[Theorem III.2.6]{Koranyi} that these symmetric spaces admit a bounded domain realization, that means that they are biholomorphic to a bounded convex subspace of $\C^n$ on which $G$ acts via biholomorphisms. An Hermitian symmetric space is said to be \emph{of tube-type} if it is also biholomorphic to a domain of the form $V+i\Om$ where $V$ is a real vector space and $\Om\subset V$ is a proper convex open cone. Hermitian symmetric spaces were classified by Cartan \cite{Cartan-class}, and are the symmetric spaces associated to the exceptional Lie groups $\E_7(-25)$ and $\E_6(-14)$ together with 4 families of classical domains: the 
ones 
associated to $\SU(p,q)$, of type $I_{p,q}$ in the standard terminology\footnote{In Cartan's original terminology \cite{Cartan-class} the families $III_p$ and $IV_p$ are exchanged}, the ones associated to $\SO^*(2p)$, of type $II_{p}$, the symmetric spaces, $III_p$, of the groups $\Sp(2p,\R)$, and the symmetric spaces of the group $IV_p$ associated to $\SO_0(2,p)$. It is well known that the only spaces that are not of tube type are the symmetric space of $\E_6(-14)$ and the families $I_{p,q}$ with $q\neq p$ and $II_p$ with $p$ odd. It follows from the classification that any Hermitian symmetric space contains maximal tube-type subdomains, and those  are all conjugate under the $G$-action,  are isometrically and holomorphically embedded and have the same rank as the ambient symmetric space.

The $G$-action via biholomorphism on the bounded domain realization of $\Xx$ extends continuously on the topological boundary $\partial \Xx$. If the real rank of $G$ is greater than or equal to two, $\partial \Xx$ is not an homogeneous $G$-space, but contains a unique closed $G$-orbit, the \emph{ Shilov boundary} $\Ss_G$ of $\Xx$. If $\Xx$ is irreducible,  the stabilizer of any point $s$ of $\Ss_G$ is a maximal parabolic subgroup of $G$. In the reducible case, if $\Xx=\Xx_1\times\ldots\times\Xx_n$ is the de Rham decomposition in irreducible factors whose isometry group is $G_i$,  then $\Ss_G$ splits as the product $\Ss_{G_1}\times\ldots\times\Ss_{G_n}$ as well. 
Moreover when $\Yy$ is a maximal tube-type subdomain of $\Xx$, the Shilov boundary of $\Yy$ embeds in the Shilov boundary of $\Xx$.

The diagonal action of $G$ on the pairs of points $(s_1,s_2)\in \Ss_G^2$ has a unique open orbit corresponding to pairs of opposite parabolic subgroups. Two points in $\Ss_G$ are \emph{transverse} if they belong to this open orbit. Whenever a pair $(s_1,s_2)$ of transverse points of $\Ss_G$ is fixed, there exists a unique maximal tube-type subdomain $\Yy=G_T/K_T$ of $\Xx$ such that $s_i$ belongs to $\Ss_{G_T}$. In particular this implies that the Shilov boundaries of maximal tube-type subdomains define a rich incidence structure in $\Ss_G$.

Given three points in $\Ss_G$ there won't, in general,  exist a maximal tube-type subdomain $\Yy$ of $\Xx$ whose Shilov boundary contains all the three points. However it is possible to determine when this happens with the aid of the K\"ahler form. Recall that, since $\Xx$ is a Hermitian symmetric space, it is possible to define a differential two form via the formula
$$\o(X,Y)= g(X,JY)$$
where $g$ denotes  the $G$-invariant Riemannian metric normalized so that its minimal holomorphic sectional curvature is $-1$, and $J$ is the complex structure of $\Xx$.
Since $\o$ is $G$-invariant, it is closed: this is true for every $G$-invariant differential form on a symmetric space. This implies that $\Xx$ is a K\"ahler manifold and $\o$ is its K\"ahler form.
Let $\Xx^{(3)}$ denote the triples of pairwise distinct points in $\Xx$ and let us consider the function
\barr \b_{\Xx}:&\Xx^{(3)}&\to&\R\\ &(x,y,z)&\to& \frac 1{\pi }\int_{\Delta(x,y,z)}\o\earr
where we denote by $\Delta(x,y,z)$ any smooth geodesic triangle having $(x,y,z)$ as vertices. Since $\o$ is closed, Stokes theorem implies that $\b_{\Xx}$ is a well defined continuous $G$-invariant cocycle and it is proven in \cite{CO} that it extends continuously to the triples of pairwise transverse points in the Shilov boundary.
If a triple $(s_1,s_2,s_3)\in\Ss^3$ doesn't consist of pairwise transverse points, the limit of $\beta_\Xx(x_1^i,x_2^i,x_3^i)$ as $x_j^i$ approaches $s_j$ is not well defined, but Clerc proved that, restricting only to some preferred sequences (the one that converge \emph{radially} to $s_j$), it is possible to get a measurable extension  of $\b_\Xx$ to the whole Shilov boundary. The obtained extension
$\b_\Ss:\Ss^{3}_G\to \R$ is called the \emph{Bergmann cocycle}\footnote{We choose the normalization of \cite{Clerc}, the normalization chosen in \cite{DT} is such that $\b_{DT}=\pi\cdot\b_{\Ss}$, the one of \cite{Toledo} is such that $\b_{BIW}=\frac{\beta_{\Ss}}{2}$} and it is a measurable strict cocycle.
The maximality of the Bergmann cocycle detects when a triple of points is contained in the Shilov boundary of a tube-type subdomain:
\begin{prop}\label{prop:bergmann}
\begin{enumerate}
\item $\beta_\Ss$ is a strict alternating $G$-invariant cocycle with values in $[-\rk\Xx,\rk\Xx]$,
\item If $\beta_\Ss(s_1,s_2,s_3)=\rk\Xx$ then the triple $(s_1,s_2,s_3)$ is contained in the Shilov boundary of a tube-type subdomain.
\item The Bergmann cocycle is a complete invariant for the $G$ action on triples of pairwise transverse points contained in a tube type subdomain.
\end{enumerate}
\end{prop}
\begin{proof}
The first fact was proven in \cite{Clerc}, the second can be found in \cite[Proposition 5.6]{tight}, the third follows from the transitivity of the $G$-action on maximal tube type subdomains of $\Ss_G$ and \cite[Theorem 5.2]{CN}.
\end{proof}
We will call a triple $(s_1,s_2,s_3)$ in $\Ss_G^3$ satisfying $|\beta_\Ss(s_1,s_2,s_3)|=\rk (\Xx)$ a \emph{maximal} triple.

In the case where $G$ is $\SU(1,p)$, that is a finite cover of the connected component of the identity in $\Isom (\H^p_\C)$, the maximal tube-type subdomains are complex geodesics of $\H^p_\C$ and the Bergmann cocycle coincides with Cartan's angular invariant $c_p$ \cite[Section 7.1.4]{Goldman}. Following Cartan's notation we will denote by \emph{chains} the boundaries of the complex geodesics.

\subsection{Continuous (bounded) cohomology and maximal representations}
We  introduce now the concepts we will need about continuous and continuous bounded cohomology, standard references are respectively \cite{BW} and \cite{Mon}. A  quick introduction to the relevant aspects of continuous bounded cohomology can also be found in \cite{Formula}.

Throughout the section $G$ will be a locally compact second countable  group, every finitely generated group fits in this class when endowed with the discrete topology. The \emph{continuous cohomology} of $G$ with real coefficients, ${\rm H}^n_{\rm c}(G,\R)$ is the cohomology of the complex $({\rm C}^n_{\rm c}(G,\R)^G,{\rm d})$ where 
$${\rm C}^n_{\rm c}(G,\R)=\{f:G^{n+1}\to \R|\; f\text{ is a continuous function }\},$$
the invariants are taken with respect to the diagonal action,
and the differential ${\rm d}^n:{\rm C}_{\rm c}^{n}(G,\R)\to {\rm C}^{n+1}_{\rm c}(G,\R)$ is defined by the expression
$${\rm d}^nf(g_0,\ldots,g_{n+1})=\sum_{i=0}^{n+1}(-1)^if((g_0,\ldots,\hat g_i,\ldots, g_{n+1}).$$
Similarly the \emph{continuous bounded cohomology} ${\rm H}^n_{\rm cb}(G,\R)$ of $G$ is the cohomology  of the subcomplex $({\rm C}^n_{\rm cb}(G,\R)^G,{\rm d})$ of $({\rm C}_{\rm c}^n(G,\R)^G, {\rm d})$ consisting of bounded functions.
The inclusion $i:{\rm C}^n_{\rm cb}(G,\R)^G\to {\rm C}_{\rm c}^n(G,\R)^G$ induces, in cohomology,  the so-called \emph{comparison map} $c: {\rm H}^n_{\rm cb}(G,\R)\to{\rm H}^n_{\rm c}(G,\R)$. 
The Banach norm on the cochain modules ${\rm C}^n_{\rm cb}(G,\R)$ defined by 
$$\|f\|_\infty=\sup_{(g_0,\ldots, g_n)\in G^{n+1}}|f(g_0,\ldots,g_n)|$$
induces a seminorm on ${\rm H}^n_{\rm cb}(G,\R)$ that is usually referred to as the \emph{canonical seminorm} or \emph{Gromov's norm}.

Most of the results about continuous and continuous bounded cohomology are based on the functorial approach to the study of these cohomological theories that is classical in the case of continuous cohomology and was developed by  Burger and Monod \cite{BMJEMS} in the setting of continuous bounded cohomology. This allows to show that the cohomology of many different complexes realizes canonically the given cohomological theory. Since we will only need applications of this machinery that are already present in the literature we will not describe it any further here and we refer instead to \cite{BW, Mon} for details on this nice subject.

 A first notable application of this approach to continuous cohomology is van Est Theorem \cite{vanEst, Dupont} that realizes the continuous cohomology of a semisimple Lie group in terms of $G$-invariant differential forms on the associated symmetric space:
\begin{thm}[van Est]
Let $G$ be a semisimple Lie group without compact factors, then
$$ \Omega^n(\Xx,\R)^G\cong {\rm H}^n_{\rm c}(G,\R) .$$
Under this isomorphism the differential form $\omega$ corresponds to the class of the cocycle $c_\o$ defined by the formula
$$c_\o(g_0,\ldots,g_n)=\frac{1}{\pi}\int_{\Delta(g_0x,\ldots g_nx)}\o$$
for any fixed basepoint $x$ in $\Xx$.
\end{thm}

Let us now focus more specifically on the second bounded cohomology of a Hermitian Lie group $G$.
By van Est isomorphism the module ${\rm H}^2_{\rm c}(G,\R)$ is isomorphic to the vector space of the $G$-invariant differential 2-forms on $\Xx$ which are generated, as a real vector space, by the K\"ahler classes of the irreducible factors of the symmetric space $\Xx$. The class corresponding via van Est isomorphism to the K\"ahler class $\o$ of $\Xx$ is represented by the cocycle $c_\o(g_0,g_1,g_2)=\b_\Xx(g_0x,g_1x,g_2x)$ where $x\in\Xx$ is any fixed point. 

It was proven in \cite{DT} for the irreducible classical domains and in \cite{CO} in the general case that the absolute value of the cocycle $c_\o$ is bounded by $\text{rk}(\Xx)$, hence the class $[c_\o]$ is in the image of the comparison map $c:{\rm H}^2_{\rm cb}(G,\R)\to {\rm H}^2_{\rm c}(G,\R)$.  
Moreover, if $G$ is a connected semisimple Lie group 
with finite center and without compact factors, the 
comparison map $c$ is injective (hence an 
isomorphism) 
in degree 2 \cite{BMJEMS}. We will denote by $\k^b_G$ the \emph{bounded 
K\"ahler class}, that is the class in ${\rm H}^2_{cb}(G,\R)$ satisfying $c(\k^b_G)=[c_\o]$. The Gromov norm of $\k^b_G$ can 
be computed explicitly:
\begin{thm}[\cite{DT,CO,tight}]\label{thm:gromovnorm}
 Let $G$ be a Hermitian Lie group with associated symmetric space $\Xx$ and let $\k^b_G$ be its bounded K\"ahler class. If $\|\cdot\|$ denotes the Gromov norm, then
 $$\|\k^b_G\|=\text{\emph{rk}}(\Xx).$$
\end{thm}

Let now $M$ be a locally compact second countable topological group,  $G$ a Lie group of Hermitian type, $\rho:M\to G$ a continuous homomorphism. The precomposition with $\rho$ at the cochain level induces a pullback map in bounded cohomology $\rho_b^*:{\rm H}^2_{\rm cb}(G,\R)\to {\rm H}^2_{\rm cb}(M,\R)$ that is norm non increasing. 
\emph{Tight homomorphisms} were first defined in \cite{tight}, these are homomorphisms $\rho$  for which the pullback map is norm preserving, namely $\|\rho^*(\k_M^b)\|=\|\k_M^b\|$. In the same paper the following structure theorem is proven:
\begin{thm}[{\cite[Theorem 7.1]{tight}}]\label{thm:tight}
 Let $L$ be a locally compact second countable group, $\mathbf G$ a connected algebraic group defined over $\R$ such that $G=\mathbf G(\R)^\circ$ is of Hermitian type. Suppose that $\rho:L\to G$ is a continuous tight homomorphism. Then 
 \begin{enumerate}
  \item The Zariski closure $\mathbf H=\ov{\rho(L)}^Z$ is reductive.
  \item The group $H=\mathbf H(\R)^\circ$ almost splits as a product $H_{nc}H_c$ where $H_c$ is compact and $H_{nc}$ is of Hermitian type.
  \item If $\Yy$ is the symmetric space associated to $H_{nc}$, then the inclusion of $\Yy$ in $\Xx$ is totally geodesic and the Shilov boundary $\Ss_{H_{nc}}$ sits as a subspace of $\Ss_G$.
 \end{enumerate}
\end{thm}
In the case when also $L$ is an Hermitian Lie group, tight homomorphisms can be completely classified: in fact it is possible to prove that, if $L$ has no simple factor locally isomorphic to $\SU(1,1)$, then the map $\rho$ is equivariant with an holomorphic map (see \cite{Ham2} for the case when $L$ is simple, and \cite{HamP} for the general case) and in particular the classification of \cite{Ham1} applies. In our setting this implies the following;
\begin{thm}\label{thm:tightol}
 Let $i:L\to \SU(m,n)$ be a tight homomorphism, assume that no factor of $L$ is locally isomorphic to $\SU(1,1)$. Then each simple factor of $L$ is either isomorphic to $\SU(s,t)$ or is of tube type. Moreover if $L=L_t\times L_{nt}$ where $L_t$ is the product of all the irreducible factors of tube type, then there exists an orthogonal decomposition $\C^{m,n}=\C^{k,k}\oplus\C^{m-k,n-k}$ such that $L_{t}$ is included in $\SU(\C^{k,k})$ and $L_{nt}$ is included in $\SU(\C^{m-k,n-k})$. 
\end{thm}

\begin{proof}
This can be found in \cite{HamP}.
\end{proof}

A key feature of bounded cohomology is that, whenever $\G$ is a lattice in $G$, it is possible to construct a left inverse ${\rm T}_{\rm b}^\bullet:{\rm H}^\bullet_{\rm b}(\G)\to{\rm H}^\bullet_{\rm cb}(G)$ of the restriction map. 
Indeed the bounded cohomology of $\G$ can be computed from the complex $({\rm C_{cb}^\bullet}(G,\R)^\G,{\rm d})$ and the transfer map ${\rm T}_{\rm b}^\bullet$ can be defined on the cochain level by the formula
$${\rm T}^k_{\rm b}(c(g_0,\ldots,g_k))=\int_{\G\backslash G}c(gg_0,\ldots,gg_k)\text{d}\mu(g)$$ where $\mu$ is the  measure on $\G\backslash G$ induced by the Haar measure of $G$ provided it is normalized to have total mass one.
It is worth remarking that when we consider instead continuous cohomology (without boundedness assumptions), a transfer map can be defined with the very same formula only for cocompact lattices, but the restriction map is in general not injective if the lattice is not cocompact.

Let us fix a representation $\rho:\G\to G$. Since ${\rm H}^2_{\rm cb}(\SU(1,p),\R)=\R\k^b_{\SU(1,p)}$, the class ${\rm T}^*_{\rm b}\rho^*(\k^b_G)$ is a scalar multiple of the K\"ahler class $\k^b_{\SU(1,p)}$ . The \emph{generalized Toledo invariant}\footnote{The original definition of the generalized Toledo invariant given in \cite{BIpreprint} used continuous cohomology only. However it is proven in  \cite[Lemma 5.3]{MW} that the invariant that was originally defined in \cite{BIpreprint}, $i_\rho$ in the notation of that article, and the invariant we defined here, that there was denoted by $\rm t_{ b}(\rho)$, coincide.} of the representation $\rho$ is the number $i_\rho$ such that ${\rm T}^*_{\rm b}\rho^*(\k^b_G)=  i_\rho\k^b_{\SU(1,p)}$. A consequence of Theorem \ref{thm:gromovnorm}, and the fact that the transfer map is norm non-increasing, is that $|i_\rho|\leq \text{rk}(\Xx)$. 
\begin{defn}
A representation $\rho$ is \emph{maximal} if  $|i_\rho|=\text{rk}(\Xx)$. Clearly maximal representation are in particular tight representations.
\end{defn}

The following lemma will be useful at the very end of the article, in the proof of Corollary \ref{cor:local rigidity}:

\begin{lem}\label{lem:Toledo constant}
 The generalized Toledo invariant is constant on connected components of the representation variety.
\end{lem}
\begin{proof}
This is proven in \cite[Page 4]{BICartan}.
\end{proof}
\subsection{Boundary maps and maximal representations}
The existence of measurable boundary maps for Zariski dense homomorphisms in algebraic groups is by now classical:
\begin{prop}\cite[Proposition 7.2]{SUpq}\label{prop:boundary map}
 Let $\G$ be a lattice in $\SU(1,p)$, $G$ a Lie group of Hermitian type and let $\rho:\G\to G$ be a Zariski dense representation. Then there exists a $\rho$-equivariant measurable map $\phi:\deH^p\to\Ss_G$  such that, for almost every pair of points $x, y$ in $\deH^p$, $\phi(x)$ and $\phi(y)$ are transverse.
\end{prop}

Let us now fix a measurable map $\phi:\deH^p\to\Ss_{G}$, and define the \emph{essential Zariski closure} of $\phi$ to be the minimal Zariski closed subset $V$ of $\Ss_{G}$ such that $\mu(\phi^{-1}(V))=1$.
Such a set exists since the intersection of finitely many closed subset of full measure has full measure and  $\Ss_{G}$ is an algebraic variety, in particular it is Noetherian.
We will say that a measurable boundary map $\phi$ is \emph{Zariski dense} if its essential Zariski closure is the whole $\Ss_{G}$.

\begin{prop}\label{prop:phiZariskidense}
 Let $\rho$ be a Zariski dense representation, then $\phi$ is Zariski dense.
\end{prop}
\begin{proof}
 Indeed let us assume by contradiction that the essential Zariski closure of $\phi(\deH^p)$ is a proper Zariski closed subset $V$ of $\Ss_G$. The set $V$ is $\rho(\G)$-invariant: indeed for every element $\gamma$ in $\Gamma$, we get $\mu(\phi^{-1}(\rho(\gamma)V))=\mu(\gamma\phi^{-1}(V))=1$, hence, in particular,  $\rho(\g)V=V$ by minimality of $V$. 
 
 Let us now recall that the Shilov boundary $\Ss_G$ is an homogeneous space for $G$, and let us fix the preimage $W$ of $V$ under the projection map $G\to G/Q=\Ss_G$. $W$ is a proper Zariski closed subset of $G$, moreover if $g$ is any element in $W$, the Zariski dense subgroup $\rho(\G)$ of $G$ is contained in $Wg^{-1}$ and this gives a contradiction.
\end{proof}

One of the advantages of bounded cohomology when proving rigidity statements is that the bounded cohomology of a group can be computed from a suitable boundary of the group itself, for example when $\G$ is a lattice in $\SU(1,p)$, the complex $({\rm L}^\infty_{\text{alt}}((\deH^p)^\bullet,\R)^\G,d)$ realizes isometrically the bounded cohomology of $\G$  \cite{BMGAFA}. Moreover, exploiting functoriality properties of bounded cohomology, one can implement the pullback via a measurable boundary map provided by Proposition \ref{prop:boundary map} thus getting the following result:

\begin{prop}[{\cite[Theorem 2.41]{Formula}} ]\label{prop:Formula}
Let $\G$ be a lattice in $\SU(1,p)$ and let $G$ be a Hermitian Lie group. Let $\rho:\G\to G$ be a representation, $\beta_\Ss:( \Ss_G)^{3}\to \R$ the Bergmann cocycle and  $\phi:\deH^p\to G/Q$ be a measurable $\rho$-equivariant boundary map.  For almost every triple $(x,y,z)$ in $\deH^p$, the formula
$$i_\rho c_p(x,y,z)=\int_{\G\backslash \SU(1,p)}\beta_\Ss(\phi(gx),\phi(gy),\phi(gz)) {\rm d}\mu(g) $$ holds.
\end{prop}

We will now show that, since $\beta_\Ss$ is a strict $G$-invariant cocycle and $\SU(1,p)$ acts transitively on pairs of distinct points of $\deH^p$, the equality holds for every triple of pairwise distinct points (this is an adaptation in our context of an argument due to Bucher: cfr. the proof of \cite[Proposition 3]{Mostow} in case $n=3$). 
\begin{lem}\label{lem:everytriple}
The equality in Proposition \ref{prop:Formula}  holds for every triple $(x,y,z)$ of pairwise distinct points.
\end{lem}
\begin{proof}
The formula of Proposition \ref{prop:Formula} is an equality between $\SU(1,p)$-invariant strict cocycles: 
clearly this is true for the left-hand side, moreover the expression on the right-hand side is a strict cocycle since $\beta_\Ss$ is, and is $\SU(1,p)$ invariant since $\phi$ is $\rho$-equivariant and $\beta_\Ss$ is $G$-invariant.

Let us now fix a $\SU(1,p)$-invariant full measure set $\Oo\subseteq (\deH^p)^{3}$ on which the equality holds. 
Since $\Oo$ is of full measure, an application of Fubini's Theorem is that for almost every pair $(y_1,y_2)\in(\deH^p)^2$ the set of points $z\in\deH^p$ such that $(y_1,y_2,z)\in\Oo$ is of full measure. Let us fix a pair $(y_1,y_2)$ for which this holds and denote by $\Ww$ the set of points $z$ such that $(y_1,y_2,z)\in\Oo$. 

Let us now fix a triple $(x_1,x_2,x_3)$ of points in $\deH^p$.
Since the $\SU(1,p)$ action on $\deH^p$ is transitive on pairs of distinct points, for every $i$ there exist an element $g_i$ such that $(x_i,x_{i+1})=(g_iy_1,g_iy_2)$. 
Let us now fix a point $x_3$ in the full measure set $g_1\Ww\cap g_2\Ww\cap g_3\Ww$. Since $x_3$ is in $g_i\Ww$, we get that $g_i^{-1}x_3\in\Ww$, and hence $(x_i,x_{i+1},x_3)=g_i(y_1,y_2,g_i^{-1}{x_3})\in \Oo$.
 In particular, computing the cocycle identity on the 4tuple $(x_0,x_1,x_2,x_3)$ we get that the identity of  Proposition \ref{prop:Formula} holds for the triple $(x_0,x_1,x_2)$.

\end{proof}

\begin{cor}\label{cor:incidence preserved}
 Let $\rho:\G\to G$ be a maximal representation and let $\phi:\deH^p\to\Ss_{G}$ be a $\rho$-equivariant measurable boundary map. Then for almost every maximal triple $(x,y,z)\in (\deH^p)^3$, the triple $(\phi(x),\phi(y),\phi(z))$ is contained in the Shilov boundary of a tube-type subdomain and is a maximal triple. 
\end{cor}
\begin{proof}
Let us fix a positively oriented triple $(x,y,z)$ of points on a chain. We know from Lemma \ref{lem:everytriple} that the equality  
$$\int_{\SU(1,p)/\G}\beta_\Ss(\phi(gx),\phi(gy),\phi(gz)) dg=\text{rk}(\Xx)$$
holds:
since $\rho$ is maximal, then $i_\rho={\rm rk} (\Xx)$, and since $(x,y,z)$ are on a chain then $c_p(x,y,z)=1$.

Since $\|\beta_\Ss\|_\infty={\rm rk}(\Xx)$,
it follows that $\beta_\Ss(\phi(gx),\phi(gy),\phi(gz))=\text{rk}(\Xx)$ for almost every $g$ in $\SU(1,p)$. By Proposition \ref{prop:bergmann}, this implies that for almost every $g\in \SU(1,p)$, the triple $(\phi(gx),\phi(gy),\phi(gz))$ is contained in the boundary of a tube-type subdomain. Since maximal triples in $\deH^p$ form  an $\SU(1,p)$-orbit,  the fact that the result holds for almost every element $g$ implies that the result holds for almost every triple of positively oriented points  in a chain. The same argument applies for negatively oriented triples.
\end{proof}

\begin{defn} A measurable map $\phi$ \emph{preserves the chain geometry} if, for almost every pair $x,y$ in $\deH^p$, the images $\phi(x),\phi(y)$ are transverse subspaces and, for  almost every maximal triple $(x,y,z)\in (\deH^p)^3$, the triple $(\phi(x),\phi(y),\phi(z))$ is maximal.
\end{defn}
This amounts to saying that the map $\phi$ induces an almost everywhere defined morphism $(\phi,\hat \phi)$ from the geometry $\deH^p\times\Cc$ whose points are points in $\deH^p$ and whose lines are the chains, to the geometry $\Ss_G\times \Tt$ whose points are points in $\Ss_{G}$ and whose lines are the Shilov boundaries of maximal tube-type subdomains of $\Ss_{G}$. The morphism $(\phi,\hat\phi)$ has the property that it preserves the incidence structure almost everywhere. 

Purpose of the next sections is to show that a measurable Zariski dense map $\phi:\deH^p\to \Ss_{\SU(m,n)}$ that preserves the chain geometry coincides almost everywhere with an algebraic map.

\section{The Chain geometry of $\Ss_{m,n}$}\label{sec:S_mn}

For the rest of the paper we will restrict our attention to the Hermitian Lie group $\SU(m,n)$ consisting of complex matrices that preserve a non-degenerate Hermitian form of signature $(m,n)$ with $m\leq n$, and denote, for the sake of brevity, by $\Ss_{m,n}$ the Shilov boundary of $\SU(m,n)$ that was previously denoted by $\Ss_{\SU(m,n)}$. The purpose of this section is to understand some features of the incidence structure of the subsets of $\Ss_{m,n}$ that arise as Shilov boundaries of the maximal tube-type subdomains of the symmetric space associated to $\SU(m,n)$. For reasons that will be explained later we will call $m$-\emph{chains} such subsets. The main tool that we will introduce in our investigations is a projection map $\pi_x$, depending on the choice of a point $x\in \Ss_{m,n}$. The map $\pi_x$ associates to a point $y$ that is transverse to $x$ the uniquely determined $m$-chain that contains both $x$ and $y$. The central results of the section are Proposition \ref{prop:lifts of k-vertical 
chains} and Proposition \ref{prop:errors}.

\subsection{A model for $\Ss_{m,n}$.}\label{sec:S_mn1}
Throughout the paper we will realize $\SU(m,n)$ as the subgroup of $\SL(m+n,\C)$ that preserves the Hermitian form $h$ represented with respect to the standard basis by the matrix $h=\bsm0&0&\Id_m\\0&-\Id_{n-m}&0\\\Id_m&0&0\esm$. 
We will denote by $\Gr_m(\C^{m,n})$ the Grassmannian of $m$-dimensional subspaces of $\C^{m+n}$.
It is well known that the Shilov boundary $\Ss_{m,n}$ can be realized as the subset of $\Gr_m(\C^{m,n})$ consisting of subspaces that are isotropic for the form $h$. Both $\Ss_{m,n}$ and the action of $\SU(m,n)$ on it are real algebraic\footnote{More details can be found in \cite{Pthesis} and in \cite{SUpq} where it is also possible to find the description of a complex variety $\mathbf{\Ss_{m,n}}$ such that $\Ss_{m,n}=\mathbf{\Ss_{m,n}}(\R)$}.

It is classical and easy to verify that the unique open $\SU(m,n)$-orbit  on $\Ss_{m,n}^2$ consists of pairs of points whose underlying vector spaces are \emph{transverse}. In particular we will often identify a point $x\in \Ss_{m,n}$ with its underlying vectorspace, and we will use the notation $x\tra y$, that would be more suited for the linear setting, with the meaning that the pair $(x,y)$ is a pair of transverse points.
We will use the notation $\Ss_{m,n}^{(2)}$ for the set of pairs of transverse points, and we will denote the set of points in $\Ss_{m,n}$ that are transverse to a given point $x$ by
$$\Ss_{m,n}^x:=\{y\in\Ss_{m,n}|\;y\tra x\}.$$
If $x$ and $y$ are transverse point, the linear span $\<x,y\>$ is a $2m$ dimensional subspace $V$ of $\C^{m+n}$ on which $h$ has signature $(m,m)$. The Shilov boundary of the maximal tube type subdomain containing $x$ and $y$ is the set
$$\Ss_{V}=\{z\in \Gr_m(V)|\;h|_z=0\}=\Ss_{m,n}\cap \Gr_m(V).$$
\begin{defn}
An $m$-\emph{chain} in $\Ss_{m,n}$ is a subspaces of the form $\Ss_V$ for some linear subspace $V$ of $\C^{m,n}$ on which $h$ restricts to an Hermitian form of signature $(m,m)$. 
\end{defn}
Clearly any pair of transverse points $x,y$ in $\Ss_{m,n}$ uniquely determines a $m$-chain $\Ss_{\langle x,y\rangle}$ with the property that both $x$ and $y$ belong to $\Ss_{\langle x,y\rangle}$. We will denote by $T_{x,y}$ such a chain.

In the case $m=1$, that is $\Xx_{m,n}=\H^n_\C$, the 1-chains are boundaries of complex geodesics or \emph{chains} in Cartan's terminology. This is the reason why we chose to call the Shilov boundaries of maximal tube-type subdomains \emph{$m$-chains}. To be more consistent with Cartan's notation, we omit the 1, and simply call \emph{chains} the $1$-chains.
\subsection{The Heisenberg model $\Hh^{m,n}(x)$.}\label{ssec:Heisenberg}
We now want to give a model for the Zariski open subset of $\Ss_{m,n}$ consisting of points transverse to a given point $x$. The model we are introducing is sometimes referred to as (the boundary of) a Siegel domain of genus two and was studied, for example, by Koranyi and Wolf in \cite{KorWo}. In the case $m=1$ this model is described in \cite[Chapter 4]{Goldman} but our conventions here will be slightly different.

In the rest of the paper, for each complex matrix $X$ we will denote by $X^T$ the transpose of $X$, by $X^*=\ov X^T$ the transpose conjugate of $X$. If, moreover, $X$ is invertible we will denote by $X^{-1}$ the inverse of $X$ and by $X^{-*}$ the inverse of $X^*$. Moreover we will indicate a point $V$ in the Grassmannian $\Gr_m(\C^{m+n})$ with a $(n+m)\cdot m$ dimensional matrix: we will understand such a matrix as an ordered basis of the subspace $V$. Clearly two matrices $X,Y$ represent the same element in $\Gr_m(\C^{m+n})$ if and only if there exists a matrix $G\in \GL_m(\C)$ such that $X=YG$. A direct computation gives that a point $x\in \Gr_m(\C^{m+n})$ represented by the matrix $\bsm X_1\\X_2\\X_3\esm$ belongs to $\Ss_{m,n}$ if and only if $X_1^*X_3+X_3^*X_1-X_2^*X_2=0$ where $X_1$ and $X_3$ have $m$ rows and $X_2$ has $n-m$ rows.

Let us focus on the maximal isotropic subspace  
$$v_\infty=\<e_i|\;1\leq i\leq m\>=\bsm\Id_m\\0\\0\esm\in \Ss_{m,n}.$$ 
The set of points $\Ss_{m,n}^{v_\infty}$ that are transverse to $v_\infty$ admit a basis of the form $\bsm X\\Y\\\Id_m\esm$ with $X^*+X-Y^*Y=0$. We will identify such a set with the linear space $M((n-m)\times m,\C)\times \fru(m)$ where $\fru(m)$ is the set of antiHermitian matrices. We use the symbol $\Hh^{m,n}(v_\infty)$ for such a linear space, that will be understood as  parametrizing $\Ss_{m,n}^{v_\infty}$ via the map
$$\begin{array}{ccccc}
  \Hh^{m,n}(v_\infty)=& M((n-m)\times m,\C)\times \fru(m)&\to &\Ss_{m,n}^{v_\infty}\\
   &(X,Y)&\mapsto &\bsm Y+X^*X/2\\X\\\Id_m\esm.
  \end{array}
$$
We refer to $\Hh^{m,n}(v_\infty)$ as the Heisenberg model. This is because, as we will now see, $\Hh^{m,n}(v_\infty)$ identifies with the generalized Heisenberg group that is the nilpotent radical of the stabilizer of $v_\infty$.

Let us denote by $Q$ the maximal parabolic subgroup of $\SU(\C^{m+n}, h)$ that is the stabilizer of $v_\infty$. It is easy to verify that
$$Q=\left\{\bpm A&B&E\\0&C&F\\0&0&A^{-*}\epm\hspace{-10pt}\begin{array}{l}_m\\_{n-m}\\_{m}                                                                                                                                                                                               \end{array} \left| 
\begin{array}{l} A\in \GL_m(\C),\, C\in \U(n-m),\\A^{-1}B-F^*C=0 \\E^*A^{-*}+A^{-1}E-F^*F=0\\\det C\det A\det A^{-*}=1\end{array} \right\}. \right. $$

The group $ Q$ can be written as $L\ltimes N$ where 
$$ L=\left\{ \left.\bpm A&0&0\\0&C&0\\0&0&A^{-*}\epm \right| \begin{array}{l}A\in \GL_m(\C), \\C\in \U(n-m)\\\det C\det A\det A^{-*}=1\end{array}\right\}$$
 is reductive and 
$$N=\left\{ \left.\bpm \Id&E^*&F\\0&\Id&E\\0&0&\Id\epm\right| F^*+F-E^*E=0\right\}$$
is  nilpotent.

 The group $L$ is the stabilizer of the two transverse points $v_\infty$ and $v_0=\<e_{n+1},\ldots,e_{n+m}\>$ of $\Ss_{m,n}$. If we denote by $a$ the determinant of the matrix $A$, an explicit isomorphism between $\GL_m(\C)\times \SU(n-m)$ and $L$ is given by:
 $$\begin{array}{ccc}
  \GL_m(\C)\times SU(n-m)&\to &L\\
(A,B)&\mapsto &\bsm A&0&0\\0&\ov aa^{-1}B&0\\0&0&A^{-*}\esm.
  \end{array}
$$
Similarly the 2 step nilpotent group $N$ can be identified with $ M(n\times m,\C)\ltimes \fru(m)$:
  $$\begin{array}{ccc}
    M((n-m)\times m,\C)\ltimes \fru(m)&\to &N\\
(E,F)&\mapsto &\bsm \Id&E^*&F+EE^*/2\\0&\Id&E\\0&0&\Id\esm.
  \end{array}
$$

It is particularly easy to describe  the action of $L$ and $N$ on $\Hh^{m,n}(v_\infty)$: the group $N$ acts by left multiplication according to the group structure on  $ M(n\times m,\C)\ltimes \fru(m)$
$$(E,F)\cdot(X,Y)=\left(E+X,F+Y+\frac{E^*X-X^*E}2\right)$$
and $L$ acts via right-left matrix multiplication on the first factor and conjugation on the second:
$$(A,B)\cdot(X,Y)=(\ov a a^{-1}BXA^{*},AYA^*).$$

\subsection{The projection $\pi_x$}\label{ssec:projpx}
We consider the projection on the first factor $\pi_{v_\infty}:\Hh^{m,n}(v_\infty)\to M((n-m)\times m,\C)$. Under the natural identification $\Hh^{m,n}(v_\infty)\cong N$, this projection corresponds to the group homomorphism whose kernel is the center $\fru(m)$ of $N$. Purpose of this section is to give a geometric interpretation of the quotient space $M((n-m)\times m,\C)$: it corresponds to a parametrization of the space of $m$-chains through the point $v_\infty$. 
In order to make this statement more precise let us consider the set 
$$ W_{v_\infty}=\{V\in \Gr_{2m}(\C^{m+n})|\; v_\infty< V,\; h|_{V} \text{ has signature } (m,m) \}.$$
The following lemma gives an explicit identification of $ W_{v_\infty}$ with the quotient space $M((n-m)\times m,\C)=N/_{\fru(m)}$:
\begin{lem}\label{lem:vchains}
 There exists a bijection between $M((n-m)\times m,\C)$  and $ W_{v_\infty}$  defined by the formula 
 \barr i:M((n-m)\times m,\C)&\to& W_{v_\infty}\\
 A&\mapsto&\bpm A^*\\\Id\\0\epm^\bot.
 \earr
 \end{lem}
\begin{proof}
Let $V$ be a point in $W_{v_\infty}$. Then $V^\bot$ is a $(n-m)$ dimensional subspace of $\C^{m+n}$ that is contained in $v_\infty^\bot$. This implies that $V^\bot$ admits a basis of the form $\bpm A&B&0\epm^T$ where $A$ has $m$ rows and $n-m$ columns and $B$ is a square $n-m$ dimensional matrix. Since the restriction of $ h$ on $V$ has signature $(m,m)$, the restriction of $ h$ to $V^\bot$ is negative definite, in particular the matrix $B$ must be invertible. This implies that, up to changing the basis of $V^\bot$, we can assume that $B=\Id_{n-m}$. This gives the desired bijection. 
\end{proof}

$W_{v_\infty}$ parametrizes the $m$-chains containing the point $v_\infty$. We will call them \emph{vertical} chains: the intersection $T^{v_\infty}$ of  a vertical chain $T$ with the Heisenberg model $\Hh_{m,n}(v_\infty)$ consists precisely of a fiber of the projection on the first factor in the Heisenberg model: 

\begin{lem}\label{lem:vertical chains}
 Let $T\subset  \Ss_{m,n}$ be a vertical chain and $V$ be its associated linear subspace. If we denote by  $p_T$ in $M((n-m)\times m,\C)$, the point $p_T=i^{-1}(V)$, we have:
 \begin{enumerate}
 \item for every $x$ in $T^{v_\infty}$, then $\pi_{v_\infty}(x)=p_T$,
 \item $T^{v_\infty}=\pi_{v_\infty}^{-1}(p_T)$,
 \item the center $ M$ of $N$ acts simply transitively on $T^{v_\infty}$.
 \end{enumerate}
 \end{lem}
\begin{proof}
 (1) An element $w$ of $\Hh_{m,n}(v_\infty)$ with basis  $\bpm X&Y& \Id_m\epm^T$  belongs to the chain $T$ if and only if $Y=p_T$: indeed the requirement that $V^\bot$ is contained in $w^\bot$ restates as
 $$0=\bpm X^*&Y^*&\Id_m\epm \bpm0&0&\Id\\0&-\Id&0\\\Id&0&0\epm\bpm p_T^*\\\Id\\0\epm=p_T^*-Y^*.$$
 This implies that for every $w\in T$, we have $\pi_{v_\infty}(w)=p_T$.
 
 Viceversa if $\pi_{v_\infty}(w)=p_T$, then $w$ is contained in $V$ and this proves (2).
 
 (3) The fact that $M$ acts simply transitively on $T^{v_\infty}$ is now obvious: indeed $M$ acts on the Heisenberg model by vertical translation stabilizing every vertical chain.
\end{proof}

The stabilizer $Q$ of $v_\infty$ naturally acts on the space $W_{v_\infty}$ and it is easy to deduce explicit formulae for this action from the formulae given in Section \ref{ssec:Heisenberg}.
In the sequel, when this will not cause confusion, we will identify $W_{v_\infty}$ with $M((n-m)\times m,\C)$ considering implicit the map $i^{-1}$. 
It is worth remarking that everything we did so far doesn't really depend on the choice of the point $v_\infty$, and a map $\pi_x: \Ss_{m,n}^{\; x}\to W_x$ can be defined for every point $x\in \Ss_{m,n}$. We decided to stick to the point $v_\infty$, since the formulae in the explicit expressions are easier.
\subsection{Projections of chains}\label{ssec:otherchains}

We now want to understand what are the possible images under the map $\pi_{v_\infty}$ of other chains. We define the \emph{intersection index} of an $m$-chain $T$ with a point $x\in \Ss_{m,n}$ by
$$i_{x}(T)=\dim (x\cap V_T)$$
where $V_T$ is the $2m$ dimensional linear subspace of $\C^{m+n}$ associated to $T$.
Clearly $0\leq i_{v_\infty}(T)\leq m$, and $i_{v_{\infty}}(T)=m$ if and only if the chain $T$ is vertical. In general we will call \emph{$k$-vertical} a chain whose intersection index is $k$: with this notation  vertical chains are $m$-vertical. Sometimes we will call \emph{horizontal} the chains that are 0-vertical (in particular each point in the chain is transverse to $v_\infty$).

In our investigations it will be precious to be able to relate different situations via the action of the group $ G=\SU(\C^{m+n}, h)$, under this respect the following lemma will be fundamental:
\begin{lem}\label{lem:transitivity on k vertical chains}
For every $k\in\{0,\ldots m\}$ the group $ G$ acts transitively on
\begin{enumerate}
 \item the pairs $(x,T)$ where $x\in  \Ss_{m,n}$ is a point and $T$ is an $m$-chain with $i_x(T)=k$,
 \item the triples $(x,y,T)$ where $x\tra y$, $y\in T$ and $i_x(T)=k$.
\end{enumerate}
In particular the intersection index is a complete invariant of $m$-chains up to the $\SU(m,n)$-action.
\end{lem}
\begin{proof}
We will prove directly the second statement. By transitivity of the $G$-action on the set of transverse pairs we can assume that $x=v_\infty$, $y=v_0$, in particular this reduces the proof to showing that $L$ acts transitively on the set of chains through $v_0$. It is not hard to show that the orthogonal to such a chain $T$ has a basis of the form $\bpm0&\Id_{n-m}&Z_3\epm^T$ for some matrix $Z_3$: any vector contained in the orthogonal to $v_0$ has vanishing components in $v_\infty$, moreover, since the orthogonal to a chain is positive definite we can assume that the central block is the identity up to changing the basis. Moreover it holds that $m-\rk (Z_3)=i_{v_\infty}(T)$. The statement is now obvious.
\end{proof}
As explained at the beginning of the section we want to give a parametrization of a generic chain $T$ and study the restriction of $\pi_{v_\infty}$ to $T$. In view of Lemma \ref{lem:transitivity on k vertical chains}, it is  enough to understand, for every $k$, the parametrization and the projection of a single $k$-vertical chain.
The $k$-vertical chain we will deal with is the chain with associated linear subspace
$$V_k=\<e_i,e_{j}+e_{m+j-k}+e_{n+j},v_0|\;1\leq i\leq k<j\leq m \>.$$
\begin{lem}
 $T_k$ is the linear subspace associated to a $k$-vertical chain $T_k$.
\end{lem}
\begin{proof}
 $V_k$ is a $2m$-dimensional subspace containing $v_0$. Moreover $V_k$ splits as the orthogonal direct sum
$$\begin{array}{rl}
V_k&=V_k^0\overset\bot \oplus  V_k^1=\\
&=\<e_i,e_{n+i}|\;1\leq i\leq k\>\overset \bot\oplus\langle e_{j}+e_{m+j-k}+e_{n+j},e_{n+j}|\;k+1\leq j\leq m\rangle=\\
&=\<v_\infty\cap V_k,e_{n+i}|\;1\leq i\leq k\>\overset \bot\oplus\langle e_{j}+e_{m+j-k}+e_{n+j},e_{n+j}|\;k+1\leq j\leq m\rangle.   
  \end{array}
$$
Since $v_\infty\cap V_k=\langle e_1,\ldots, e_k\rangle$, we get  that $i_{v_\infty}( T_k)$ is $k$. Since $ h|_{V_k^0}$ has signature $(k,k)$ and $ h|_{V_k^1}$ has signature $(m-k,m-k)$, we get that  the restriction of $ h$ on $V_k$ has signature $(m,m)$ and this concludes the proof.
 \end{proof}

\begin{lem}\label{lem:parametrization of k-vertical chains}
$\Hh_{m,n}(v_\infty)\cap T_k$ consists precisely of those subspaces of $\C^{m+n}$ that admit a basis of the form 
$$\begin{array}{ccc}
 \bpm \frac {E^*E}2 +C&E^*X\\
 E&\Id+X\\
 E &\Id+X\\
 0&0\\
 \Id_k & 0\\
 0&\Id_{m-k}
 \epm
 & \text{ with } & \left\{\begin{array}{l} E\in M((m-k)\times k,\C)\\
                               X\in \U(m-k)\\
                               C\in \mathfrak u(k).
 \end{array}\right.
 \end{array}$$
 The projection of $T_k$ is contained in an affine subspace of $M((n-m)\times m,\C)$ of dimension  $m^2-km$, and consists of the points of $M((n-m)\times m,\C)$ that have expression
$\bsm E&\Id+X\\0&0\esm$  with $E\in M((m-k)\times k,\C)$ and $X\in \U(m-k)$.
\end{lem}
\begin{proof}
It is enough to check that the orthogonal to $V_k$ is 
$$V_k^\bot=\< e_{m+j}+e_{n+j+k},e_{m+l}|\; 1\leq j\leq m-k<l\leq n-m\>.$$
This implies that any $m$-dimensional subspace $z$ of $V_k$, that is transverse to $v_\infty$, has a basis of the form 
$$\begin{array}{ccc}
 z=\bpm Z_{11} &Z_{12}\\
 Z_{21}&Z_{22}\\
 Z_{21} &Z_{22}\\
 0&0\\
 \Id & 0\\
 0&\Id
 \epm\hspace{-7pt}
 \begin{array}{l}
  _k\\_{m-k}
  \\_{m-k}\\_{n-2m+k}\\_{k}\\_{m-k}
 \end{array}
 & \begin{array}{l}\text{and the}\\\text{restriction }\\\text{of $ h$ to $z$ }\\\text{is zero}\\\text{if and only if }\end{array}& 
 \left\{\begin{array}{l}  Z^*_{11}+Z_{11}=Z_{21}^*Z_{21}\;(11)\\
                          Z^*_{21}+Z_{12}=Z_{21}^*Z_{22}\; (12)\\
                          Z^*_{12}+Z_{21}=Z_{22}^*Z_{21}\; (21)\\
                          Z^*_{22}+Z_{22}=Z_{22}^*Z_{22}\;(22).
 \end{array}\right.
 \end{array}$$
Equation $(22)$ restates as $Z_{22}=\Id+X$ for some $X\in \U(m-k)$:
indeed a square matrix $Z$ satisfies the equation $Z^*+Z=Z^*Z$, if and only if the equation $(Z-\Id)^*(Z-\Id)=Z^*Z-Z-Z^*+\Id=\Id$ holds, which means that $Z-\Id$ belongs to $\U(m-k)$.

This concludes the proof of the first part of the lemma: the $(m-k)\times k$ matrix $Z_{21}$ can be chosen arbitrarily, Equation $(12)$ uniquely determines $Z_{12}$ in function of $Z_{21}$ and $Z_{22}$, and Equation $(11)$ determines the Hermitian part of $Z_{11}$ in function of $Z_{21}$, but is satisfied independently on the antiHermitian part of $Z_{11}$. This proves the first part of the lemma. 

The second part is a direct consequence of the parametrization of $T_k^{v_\infty}$ we just gave, together with the identification of $W_{v_\infty}$ and $M((n-m)\times m,\C)$ given in Lemma \ref{lem:vchains}.
\end{proof}
\begin{defn}
We will call a subset of $ W_{x}$ that is the projection of a $k$-vertical chain a \emph{$(m,k)$-circle}.
\end{defn}
 
The reason for the name \emph{circle} is due to the fact that, in the case $(m,n)=(1,2)$ the projections of horizontal chains are Euclidean circles in $\C$. This fact was first observed and used by Cartan in \cite{Cartan}.  In fact every Euclidean circle $C\subseteq \C^{p-1}$ is a circle in our generalized definition, namely is the projection of some 1-chain of $\deH^p$. Indeed we know from Lemma \ref{lem:parametrization of k-vertical chains} that the Euclidean circle $(1+e^{it},0,\ldots,0)\in \C^{p-1}$ is the projection of the chain associated to the linear subspace $\<e_1+e_2,e_{p+1}\>$ of $\C^{p+1}$. Moreover the set of Euclidean circles is a homogeneous space under the group of Euclidean similarities of $\C^{p-1}$ and the group $Q=\stab(v_\infty)$ acts on $\C^{p-1}$ as the full group of Euclidean similarities.

In the general case it is important to record both the dimension of the $m$-chain that is projected and the dimension of the $\U(m-k)$ factor in the projection. This explains  the notation. 
We will call \emph{generalized circle} any subset of $M((n-m)\times m,\C)$ arising as a projection of an $m$-chain. In particular a generalized circle is an $(m,k)$-circle for some $k$.

The ultimate goal of this section is to understand the possible lifts of a given $(m,k)$-circle. We begin by analyzing the stabilizers in $\SU(\C^{m+n}, h)$ of some configurations:
\begin{lem}\label{lem:S_0}
The stabilizer in $\SU(\C^{m+n}, h)$ of the triple $(v_\infty,v_0,T_k)$ is the subgroup $S_0$ of $L\cong \GL_m(\C)\times\SU(n-m)$ consisting of pairs of the form

$$\begin{array}{ccc}

\left(
\bpm Y& X\\0&\ov y y^{-1}C_{11}\epm, \bpm C_{11}&0\\0&C_{22}\epm

\right)
& \text{with} &
\left\{
\begin{array}{l}
 C_{11} \in \U(m-k)\\
C_{22}\in \U(n-2m+k)\\
X\in M(k\times(m-k),\C)\\
Y\in \GL_k(\C), \,y=\det(Y) .\end{array}\right.

\end{array}$$
\end{lem}
\begin{proof}
We determined in Section \ref{ssec:Heisenberg} that the stabilizer $ L$ in $\SU(\C^{m+n}, h)$ of the pair $(v_\infty,v_0)$ is isomorphic to $\GL_m(\C)\times \SU(n-m)$. The stabilizer of the triple  $(v_\infty,v_0,T_k)$ is clearly contained in $ L$ and consists precisely of the elements of $ L$ stabilizing $V_k^\bot$.

In the proof of Lemma \ref{lem:parametrization of k-vertical chains} we saw that the subspace $V_k^\bot$ has a basis of the form $\bsm0\\Id_{n-m}\\X\esm$ where $X$ denotes the $m\times (n-m)$ matrix $\bsm0&0\\Id_{m-k}&0\esm$.

From the explicit expression of elements in $L$ we get
$$\bpm A&&\\&\ov a a^{-1}C&\\&&A^{-*}\epm\bpm0\\\Id\\X\epm=\bpm0\\\ov a a^{-1}C\\A^{-*}X\epm\cong\bpm0\\\Id\\a\ov a^{-1}A^{-*}XC^{-1}\epm.$$

In turn the requirement that $a\ov a^{-1}A^{-*}XC^{-1}=X$, that is $X=\ov a a^{-1}A^*XC$, implies, in the suitable block decomposition for the matrices, that
$$\begin{array}{c}

\ov a a^{-1}\bpm A_{11}^*&A_{21}^*\\A_{12}^*&A_{22}^*\epm\bpm0&0\\\Id_{m-k}&0\epm\bpm C_{11}&C_{12}\\C_{21}&C_{22}\epm=\\

\ov a a^{-1}\bpm A_{21}^*&0\\A_{22}^*&0\epm\bpm C_{11}&C_{12}\\C_{21}&C_{22}\epm=
\ov a a^{-1}\bpm A_{21}^*C_{11}&A_{21}^*C_{12}\\A_{22}^*C_{11}&A_{22}^*C_{12}\epm.
\end{array}$$

This implies that $A_{22}^{*}= a \ov a^{-1}C_{11}^{-1}$ and $C_{12}=A_{21}=0$. 
Moreover since $C$ is unitary, also $C_{21}$ must be 0, and both $C_{22}$ and $C_{11}$ must be unitary.
 This concludes the proof.
\end{proof}

Let us now denote by $o$ the point $o=\pi_{v_\infty}(v_0)=0$ in $W_{v_\infty}$ and by $C_k$ the $(m,k)$-circle that is the projection of $T_k$. We will denote by $S_1$ the stabilizer in $ Q$ of the pair $(o,C_k)$.

\begin{lem}\label{lem:S_1}
The stabilizer of the pair $(o,C_k)$ is the group
 $$S_1=\Stab_{Q}(o,C_k)=M\rtimes S_0$$
 where, as above, we denote by $M$ the center of the nilpotent radical $N$ of $Q$ and by $S_0$  the stabilizer in $Q$ of the pair $(v_0,T_k)$.
\end{lem}
\begin{proof}
Recall that any element in $Q$ can be uniquely written as a product $nl$ where $n$ is in $N$, and $l$ belongs to $L$, the Levi component of $ Q$.
Since any element in $S_1$ fixes, by assumption, the point $o=\pi_{v_\infty}(v_0)$ and since any element in $L$ fixes $o$, if $nl$ is in $S_1$ then $n(o)=o$ that, in turn, implies that $n$ belongs to $M$. Hence $S_1$ is of the form $M\rtimes S$ for some subgroup $S$ of $L$.

Let now $X$ be a point in $W_{v_\infty}=M((n-m)\times m,\C)$. The action of $(A,C)\in  L$ on $W_{v_\infty}$ is $X\mapsto \ov a a^{-1}CXA^*$. We want to show that if $C_k$ is preserved then $(A,C)$ must belong to $S_0$. We have proven in Lemma \ref{lem:parametrization of k-vertical chains} that any point $z\in \pi_{v_\infty}(T_k)$ can be written as $\bsm E&\Id+Z\\0&0\esm$ for some matrices $E\in M((m-k)\times k,\C)$ and $Z\in \U(m-k)$. Explicit computations give that

$$\begin{array}{rl} \bpm E&\Id+Z\\0&0\epm&=\ov a a^{-1} \bpm C_{11}&C_{12}\\C_{21}&C_{22}\epm\bpm E&\Id+Z\\0&0\epm\bpm A_{11}^*&A_{21}^*\\A_{12}^*&A_{22}^*\epm=\\
&=\ov a a^{-1}\bpm C_{11}E&C_{11}(\Id+Z)\\C_{21}E&C_{21}(\Id+Z)\epm\bpm A_{11}^*&A_{21}^*\\A_{12}^*&A_{22}^*\epm.\end{array}$$
Since $A$ is invertible, $E$ is arbitrary and both $\Id$ and $-\Id$ are in $\U(m-k)$, the matrix $C_{21}$ must be zero. Hence $C$ must have the same block form of a genuine element of $S_0$. In particular $C_{11}$ is invertible.
Since $\ov a a^{-1}C_{11}(EA_{21}^*+(\Id+Z)A_{22}^*)$ must be an element of $\Id+U(m-k)$ for every $E$, we get that $A_{21}^*$ must be zero. 

The result now follows from Claim \ref{claim:fixing Um} below.
\end{proof}

\begin{claim}\label{claim:fixing Um}
Let $C\in {\rm U}(l)$ and $A\in \GL_l(\C)$ be matrices and let $\Uu$ denote the set 
$$\Uu=\{\Id+X|\;X\in \U(l)\}\subset M(l\times l,\C).$$
If $C\Uu A^*=\Uu$ then $A=C$.
\end{claim}
\begin{proof}
Let us consider the birational map 
$$\begin{array}{cccc}i:&M(l\times l,\C)&\to &M(l\times l,\C)\\ &X&\mapsto &X^{-1}\end{array}$$
that is defined on a Zariski open subset $\Oo$ of $M(l\times l,\C)$.

The image, under the involution $i$, of $\Uu$ is the set $$\Ll=\{W|\;\Id-W^*-W=0\}=\frac 12\Id+\mathfrak u(l).$$ Moreover $i( CXA^{*})=A^{-*}i(X)C^{-1}$, hence in order to show that the subgroup preserving $\Uu$ consists precisely of the pairs $(A,A)$, it is enough to check that the subgroup of $\U(l)\times \GL_l(\C)$ preserving $\Ll$ consists precisely of the pairs $(A,A)$ with $A\in \U(l)$. 

This last statement amounts to show that the only matrix $B\in \GL_l(\C)$ such that $\Id-W^*B^{*}-BW=0$ for all $W\in \Ll$ is the identity itself.
Choosing $W$ to be $\frac 12 \Id$ we get that $B^*+B=2\Id$ hence in particular $B=\Id +Z$ with $Z\in \mathfrak u(l)$. Since moreover $\Ll=\{\frac 12 \Id+M|\;M\in \mathfrak{u}(l)\}$ we have to show that if $ZM+M^*Z^*=ZM+MZ=0$ for all $M$ in $\mathfrak{u}(l)$ then $Z$ must be zero, and this can be easily seen, for example by choosing $M$ to be the matrix that is zero everywhere apart from the $l$-th diagonal entry where it is equal to $i$.  
\end{proof}

We now have all the ingredients we need to prove the first crucial result of the section. Recall from Section \ref{sec:S_mn1} that every pair of transverse points $x,y$ in $\Ss_{m,n}$ uniquely determines an $m$-chain $T_{x,y}$ that is the unique chain that contains both $x$ and $y$.
\begin{prop}\label{prop:lifts of k-vertical chains}
Let $x\in \Ss_{m,n}$ be a point, $T$ be a chain with $i_x(T)=k$, $t\in T$ be a point,  $y=\pi_{x}(t)\in W_{x}$. Then
\begin{enumerate}
 \item $T$ is the unique lift of the $(m,k)$-circle $\pi_{x}(T)$ through the point $t$,
 \item for any point $t_1$ in $T_{x,t}=\pi_x^{-1}(y)$ there exists a unique $m$ chain through $t_1$ which lifts $\pi_x(T)$.

 \end{enumerate}
\end{prop}
\begin{proof}
(1) As a consequence of Lemma \ref{lem:transitivity on k vertical chains}, in order to prove the statements, we can assume that the triple $(x,t,T)$ is the triple $(v_\infty, v_0, T_k)$. Let $T'$ be another $m$-chain containing the point $t$ that lifts the $(m,k)$-circle $C_k$, a consequence of Lemma \ref{lem:transitivity on k vertical chains} is that there exists an element $g\in L$ such that  $(v_\infty,v_0,T_k)=g(v_\infty,v_0,T')$. Moreover, since $\pi_{v_\infty}(T')=C_k$, we get that $g\in S_1$. But we know that $S_1\cap L=S_0$ and this proves that $T'=T_k$.
\newline (2) This is a consequence of the first part, together with the observation that $M$ acts transitively on the vertical chain $T_{v_0,v_\infty}$.
\end{proof}
We conclude the section by determining what are the lifts of a point $y$ that are contained in an $m$-chain $T$. For every $m$-chain $T$ we consider the subgroup
$$M_T=\Stab_{M}(T).$$
Clearly if $t$ is a lift of a point $y\in W_{v_\infty}$ that is contained in $T$, then all the orbit $M_T\cdot t$ consists of lifts of $y$ that are contained in $T$.
We want to show that also the other containment holds, namely that the lifts are precisely the $M_T$ orbit of any point. 
\begin{lem} 
For the chain $T_k$ we have $M_{T_k}=i(E_k)$ where
$$E_k=\{X\in \mathfrak u(m)|\; X_{ij}=0 \text { if } i>k \text{ or } j>k\}=\left\{\bpm X_1&0\\0&0\epm\Big|\;X_1\in\fru(k)\right\},$$
and $i:\fru(m)\to N$ is the inclusion of the center of the group.
\end{lem}
\begin{proof}
We already observed that the orthogonal to $V_k$ is $$V_k^\bot=\<e_{m+j}+e_{n+j+k},e_{l+m}|\;1\leq j\leq m-k< l\leq n-m\>.$$ Moreover an element of $M$ stabilizes $T_k$ if and only if it stabilizes $V_k^\bot$.  
If now $m=\bsm\Id&0&E\\0&\Id&0\\0&0&\Id\esm$ is an element of $M$, then the image $m\cdot(e_{m+j}+e_{n+j+k})=\sum E_{ij}e_{j+k}+e_{m+j}+e_{n+j+k}$ that belongs to $V_k^\bot$ if  and only if the $(j+k)$-th column of the matrix $E$ is zero. This implies that the subgroup of $M$ that fixes $V_k^\bot$ is contained in $i(E_k)$. Viceversa it is easy to check that $i(E_k)$ belongs to $\SU(V_k)$, in particular it preserves $T_k$.
\end{proof}

We denote  by $Z_T$ the intersection of the linear subspace $V_T$ underlying $T$ with $v_\infty$:
$$Z_T=v_\infty\cap V_T.$$
In the standard case in which $T=T_k$ we will denote by $Z_k$ the subspace $Z_{T_k}$ which equals to the span of the first $k$ vectors of the standard basis of $\C^m$.
\begin{prop}\label{prop:errors}
 Let $T$ be a $k$-vertical chain, then
 \begin{enumerate}
  \item If $g\in Q$ is such that $gT=T_k$, then $M_T=g^{-1}M_{T_k}g$.
  \item For any point $x\in T$, we have $\pi_{v_\infty}^{-1}(\pi_{v_\infty}(x))\cap T=M_Tx$.
  \item If $n\in N$, then $M_{nT}=M_T$.
  \item If $a\in \GL(m)$ is such that $a(Z_T)=Z_k$, then $M_T=i(a^{-1}E_k a^{-*})$.
 \end{enumerate}
\end{prop}
 \begin{proof}
(1) This follows from the definition of $M_{T_k}$ and $M_T$ and the fact that $M$ is normal in $Q$.
\newline (2) Let us first consider the case $T=T_k$. In this case the statement is an easy consequence of the explicit parametrization of the chain $T_k$ we gave in Lemma \ref{lem:parametrization of k-vertical chains}: any two points in $T_k$ that have the same projection are in the same $M_{T_k}$ orbit. The general case is a consequence of the transitivity of $Q$ on $k$-vertical chains: let $g\in Q$ be such that $gT=T_k$ and let us denote by $y$ the point $gx$. Then we know that 
$M_{T_k}y=\pi^{-1}_{v_\infty}(\pi_{v_\infty}(y))\cap T_k$. This implies that
$$\begin{array}{rl}
   M_Tx&= g^{-1}M_{T_k}g x=g^{-1}(M_{T_k}y)=g^{-1}( \pi^{-1}_{v_\infty}(\pi_{v_\infty}(y))\cap T_k)=\\
   &=g^{-1} \pi^{-1}_{v_\infty}(\pi_{v_\infty}(y))\cap g^{-1}T_k=\\
   &=\pi^{-1}_{v_\infty}(\pi_{v_\infty}(x))\cap T.
  \end{array}
$$
Where in the last equality we used that the $Q$ action on $\Hh_{m,n}(v_\infty)$ induces an action of $Q$ on $W_{v_\infty}$ so that the projection $\pi_{v_\infty}$ is $Q$ equivariant.
\newline (3) This is a consequence of the fact that $M$ is in the center of $N$: $M_{nT}=nM_Tn^{-1}=M_T$.
\newline (4) By (3) we can assume that $T$ is a chain through the point $v_0$: indeed there exists always an element $n\in N$ such that $nT$ contains $v_0$, moreover both $M_{nT}=M_T$ and $Z_{nT}=Z_T$ (the second assertion follows from the fact that any element in $N$ acts trivially on $v_\infty$).

Since $v_0\in T$ and we proved in Lemma \ref{lem:transitivity on k vertical chains} that $L$ is transitive on $k$-vertical chains through $v_0$, we get that there exists a pair $(C,A)\in \U(n-m)\times \GL_m(\C)$ such that, denoting by $g$ the corresponding element in $L$, we have $gT=T_k$. It follows from (1) that $M_T=g^{-1}M_{T_k} g$, in particular we have
$$\bpm A^{-1}&0&0\\0&C^{-1}&0\\0&0&A^*\epm\bpm\Id&0&E\\0&\Id&0\\0&0&\Id\epm\bpm A&0&0\\0&C&0\\0&0&A^{-*}\epm=\bpm\Id&0&A^{-1}EA^{-*}\\0&\Id&0\\0&0&\Id\epm$$
and hence the subgroup $M_T$ is the group $i(A^{-1}E_kA^{-*})$. Moreover, since $gT=T_k$ we have in particular that $gZ_T=Z_k$ and hence $A(Z_T)=Z_k$ if we consider $Z_T$ as a subspace of $v_\infty$.

In order to conclude the proof it is enough to check that for every $a\in \GL_m(\C)$ with $a(Z_T)=Z_k$ the subgroups $a^{-1}E_k a^{-*}$ coincide. Indeed it is enough to check that for every element $a\in \GL_m(\C)$ such that $a(Z_k)=Z_k$ then $a^{-1}E_k a^{-*}=E_k$. But if $a$ satisfies this hypothesis, the matrix $a^{-*}$ has the form $\bsm A_1&0\\A_2&A_3\esm$. In particular we can compute:
$$a^{-1}Xa^{-*}=\bpm A_1^*&A_2^*\\0&A_3^*\epm\bpm X_1&0\\0&0\epm\bpm A_1&0\\A_2&A_3\epm=\bpm A_1^*X_1A_1&0\\0&0\epm$$
and the latter matrix still belongs to $E_k$.
\end{proof}

\section{The restriction to a chain is rational}\label{sec:chain}
In this section we prove that the chain geometry defined in Section \ref{sec:S_mn} is rigid in the following sense:
\begin{thm}\label{thm:restriction rational}
 Let $\phi:\deH^p\to \Ss_{m,n}$ be a measurable, chain geometry preserving, Zariski dense map. Then for almost every chain $C$ in $\deH^p$ the restriction $\phi|_{C}$ coincides almost everywhere with a rational map.
\end{thm}

Let us recall that, whenever a point $x\in  \Ss_{m,n}$ is fixed, the center $M_x$ of the nilpotent radical $N_x$ of the stabilizer $Q_x$ of $x$ in $\SU(\C^{m+n}, h)$ acts on the Heisenberg model $\Hh_{m,n}(x)$. Moreover, for every $m$-chain $T$ containing the point $x$, the $M_x$ action is simply transitive on the Zariski open subset $T^x$ of $T$.
The picture above is true for both $\deH^p\cong \Ss_{1,p}$ and $\Ss_{m,n}$ where, if $x\in \deH^p$, the group $M_x$ can be identified with $\fru(1)$, and, if $\phi:\deH^p\to \Ss_{m,n}$ is the boundary map, $M_{\phi(x)}\cong \fru(m)$.

The idea of the proof is to show that, for almost every point $x\in \deH^p$, the boundary map is equivariant with respect to a measurable homomorphism $h: M_x\to  M_{\phi(x)}$. Since such homomorphism must be algebraic, we get that the restriction of $\phi$ to almost every chain through $x$ must be algebraic.

In order to define the homomorphism $h$ we will prove first that a map $\phi$ satisfying our assumptions induces a measurable map $\phi_x:W_x\to W_{\phi(x)}$. Here $W_x$ can be identified with $\C^{p-1}$ and $W_{\phi(x)}$ can be identified with $M((n-m)\times m,\C)$, both these identifications are non canonical but we fix them once and forall. We will then use the map $\phi_x$ to define a cocycle  $\alpha:M_x\times (\deH^p)^x\to M_{\phi(x)}$ with respect to which $\phi$ is equivariant. We will then show that $\alpha$ is independent on the point $x$ and hence coincides almost everywhere with the desired homomorphism.

\subsection{First properties of chain preserving maps} 
Recall from Section \ref{sec:Toledo} that a map $\phi$ is Zariski dense if the essential Zariski closure of $\phi(\deH^p)$ is the whole $\Ss_{m,n}$, or, equivalently if the preimage under $\phi$ of any proper Zariski closed subset of $\Ss_{m,n}$ is not of full measure. Moreover, by definition, a measurable boundary map \emph{preserves the chain geometry} if the image under $\phi$ of almost every  pair of distinct points is a pair of transverse points, and the image of almost every maximal triple $(x_0,x_1,x_2)$ in $(\deH^p)^3$, is contained in an $m$-chain.

We will denote by $\Tt_1$ the set of chains in $\deH^p$, and by $\Tt_m$ the set of $m$-chains of $\Ss_{m,n}$. The set $\Tt_1$ is a smooth manifold, indeed an open subset of the Grassmannian $\Gr_{2}(\C^{p+1})$, and we will endow $\Tt_1$ with its Lebesgue measure class. 

The following lemma, an application in this context of Fubini's theorem, gives the first property of a chain geometry preserving map:
\begin{lem}\label{lem:1}
Let $\phi:\deH^p\to \Ss_{m,n}$ be a chain geometry preserving map. For almost every chain $C\in\Tt_1$ there exists an $m$-chain $\hat\phi(C)\in \Tt_m$ such that, for almost every point $x$ in $C$, $\phi(x)\in\hat\phi(C)$.
\end{lem}
\begin{proof}
 There is a bijection between the set $(\deH^p)^{\{3\}}$ consisting of triples of distinct points on a chain and the set 
 $$\Tt_1^{\{3\}}=\{(C,x,y,z)|\;C\in \Tt_1\text{ and } (x,y,z)\in C^{(3)}\}.$$
 In turn the projection onto the first factor endows the manifold $\Tt_1^{\{3\}}$ with the structure of a smooth bundle over $\Tt_1$. In particular Fubini's theorem implies that, for almost every chain $C\in \Tt_1$ and for almost every triple $(x,y,z)\in C^3$,  the triple $(\phi(x),\phi(y),\phi(z))$ belongs to the same $m$-chain $\hat \phi(C)$. 
 Moreover $\hat \phi(C)$ has the desired properties again as a consequence of Fubini theorem.
 \end{proof}
 
 We can now use the fact that each pair of transverse points $a,b$ in $\Ss_{m,n}$ uniquely determines a chain $T_{a,b}$ to reformulate
Lemma \ref{lem:1} in the following way:
\begin{cor}\label{cor:map on chains}
 Let $\phi:\deH^p\to \Ss_{m,n}$ be a measurable chain preserving map. Then there exists a measurable map $\hat \phi:\Tt_1\to \Tt_m$ such that, for almost every pair $(x,T)\in \deH^p\times \Tt_1$ with $x\in T$, then $\phi(x)\in \hat\phi(T)$.  
\end{cor}
\begin{proof} 
 The only thing that we have to check is that the map $\hat \phi$ is measurable, but this follows from the fact that the map associating to a pair $(x,y)\in \Ss_{m,n}^{(2)}$ the $m$-chain $T_{x,y}$ is algebraic.
\end{proof}
Recall that, if $x$ is a point in $\deH^p$, we denote by $W_x$ the set of chains through $x$. We use the identification of $W_x$ as subvariety of $\Gr_2(\C^{1,p})$ to endow the space $W_x$ with its Lebesgue measure class.
\begin{cor}\label{cor:point-chain}
For almost every $x\in\deH^p$, almost every chain in $W_x$ satisfies Lemma \ref{lem:1}. 
\end{cor}
\begin{proof}
 It is again an application of Fubini's theorem. Let us indeed consider the manifold $\Tt^{\{1\}}=\{(C,x)|\;C\in\Tt,\,x\in C\}$ the projection on the first two factor $\Tt^{\{3\}}\to\Tt^{\{1\}}$ realizes the first manifold as a smooth bundle over the second with fiber $(\R\times\R)\backslash \Delta$. In particular for almost every pair in $\Tt^{\{1\}}$ the chain satisfies the assumption of Corollary \ref{cor:map on chains}. Since $\Tt^{\{1\}}$ is a bundle over $\deH^p$ with fiber $W_x$ over $x$, the statement follows applying Fubini again.
\end{proof}

We will call a point $x$ that satisfies the hypotheses of Corollary \ref{cor:point-chain} \emph{generic} for the map $\phi$. Let us now fix, for the rest of the section, a point $x$ that is generic for the map $\phi$ and consider the diagram
$$\xymatrix{\Hh_{1,p}(x)\ar[r]^-\phi\ar[d]^{\pi_x}&\Hh_{m,n}(\phi(x))\ar[d]^{\pi_{\phi(x)}}\\W_x\ar@{.>}[r]^{\phi_x}&W_{\phi(x)}.}$$
\begin{lem}\label{lem:phix}
 If $x$ is generic for $\phi$, there exists a measurable map $\phi_x$ such that the diagram commutes almost everywhere. Moreover $\phi_x$ induces a measurable map $\hat\phi_x$ from the set of circles of $W_x$ to the set of generalized circles of $W_{\phi(x)}$ such that, for almost every chain $T$, we have that $\hat\phi(T)$ is a lift of $\hat\phi_x(\pi_x(T))$. 
\end{lem}
\begin{proof}
The fact that a map $\phi_x$ exists making the diagram commutative on a full measure set is a direct application of Corollary \ref{cor:point-chain}.

Since the set of horizontal chains in $\deH^p$ is a smooth bundle over the set of Euclidean circles in $W_x\cong \C^{p-1}$,  we have that, for almost every Euclidean circle $C$, the map $\hat \phi$ is defined on almost every chain $T$ with $\pi_x(T)=C$. Moreover a Fubini-type argument implies that, for almost every circle $C$, the diagram commutes when restricted to the preimage of $C$.

This implies that the projections $\hat\phi_x(C):=\pi_{\phi(x)}(\hat\phi(T_i))$ coincide for almost every lift $T_i$ of $C$ if $C$ satisfies the hypotheses of the previous paragraph, and this concludes the proof.    
\end{proof}

\subsection{A measurable cocycle}\label{ssec:cocycle}
Recall that if $H,K$ are topological groups  and $Y$ is a Borel $H$-space, then a map $\alpha:H\times Y\to K$ is a \emph{Borel cocycle} if it is a measurable map such that, for every $h_1,h_2$ in $H$ and for almost every $y\in Y$, it holds $\alpha(h_1h_2,y)=\alpha(h_1,h_2\cdot y)\alpha(h_2,y)$.

\begin{prop}\label{prop:cocycle}
 Let $\phi$ be a measurable, chain preserving map $\phi:\deH^p\to\Ss_{m,n}$. For almost every point $x$ in $\deH^p$ there exists a measurable cocycle $\alpha: M_x\times (\deH^p)^{x}\to  M_{\phi(x)}$
 such that $\phi$ is $\alpha$-equivariant. 
\end{prop}
\begin{proof}
 Let us fix a point $x$ generic for the map $\phi$. For almost every pair $(e,y)$ where $e\in M_x$ and $y\in\Hh_{1,p}(x)$, we have that the points $\phi(y)$ and $\phi(ey)$ are on the same vertical chain in $\Hh_{m,n}(\phi(x))$. In particular there exists an element $\alpha(e,y)\in M_{\phi(x)}$ such that $\alpha(e,y)\phi(y)=\phi(ey)$. We extend $\alpha$ by defining it to be 0 on pairs that do not satisfy this assumption. The function $\alpha$ is measurable since $\phi$ is measurable.
 
We now have to show that the map $\alpha$ we just defined is actually a cocycle.
In order to do this let us fix the set $\Oo$ of points $z$ for which $\hat\phi(T_{x,z})$ is $m$-vertical and $\phi(z)\in \hat \phi(T_{x,z})$. $\Oo$ has full measure as a consequence of Lemma \ref{lem:1}. Let us now fix two elements $e_1,e_2\in M_x$. For every element $z$ in the full measure set $\Oo\cap e_2^{-1}\Oo\cap e_1e_2^{-1}\Oo$, the three points $\phi(z),\phi(e_2z),\phi(e_1e_2z)$ belong to the same vertical $m$-chain, moreover, by definition of $\alpha$, we have
$$\begin{array}{rl}\alpha(e_1e_2,z)\phi(z)&=\phi(e_1(e_2z))\\&=\alpha(e_1,e_2z)\phi(e_2z)\\&=\alpha(e_1,e_2z)\alpha(e_2,z)\phi(z).\end{array}$$
The conclusion follows from the fact that the action of $M_{\phi(x)}$ on $\Hh_{m,n}(\phi(x))$ is simply transitive.
\end{proof}

\begin{prop}\label{prop:homom>algebraic}
 Let us fix a point $x$. Assume that there exists a measurable function $\beta:M_x\times W_x\to M_{\phi(x)}$ such that for every $e\in M_x$, for almost every $T$ in $W_x$ and for almost every $z$ in $T$, the equality $\alpha(e,z)=\beta(e,T)$ holds. Then the restriction of the boundary map $\phi$ to almost every chain through the point $x$ is rational.
\end{prop}
\begin{proof}
We are assuming that for every $e$ in $M_x$ for almost every $T$ in $W_x$ and for almost every $z$ in $T$, the equality $\alpha(e,z)=\beta(e,T)$ holds. Fubini's Theorem then implies that for every $T$ in a full measure subset $\Ff$ of $W_x$, for almost every $e$ in $M_x$ and almost every $z$ in $T$ the equality $\alpha(e,z)=\beta(e,T)$ holds. In particular for every vertical chain $T$ in $\Ff$ and almost every pair $(e_1,e_2)$ in $M_x^2$ we have $\beta(e_1,T)\beta(e_2,T)=\beta(e_1e_2,T)$: it is in fact enough  to chose $e_1$ and $e_2$ so that the equality of $\alpha(e_i,z)$ and $\beta(e_i,T)$ holds for almost every $z$ and compute the cocycle identity for $\alpha$ in a point $z$ that works both for $e_1$ and $e_2$.

It is classical that if $\pi:G\to J$ satisfies $\pi(xy)=\pi(x)\pi(y)$ for almost every pair $(x,y)$ in $G^2$, then $\pi$ coincides almost everywhere with an actual Borel homomorphism (cfr. \cite[Theorem B.2]{Zimmer}).  In particular for every $T$ in $\Ff$, we can assume (up to modifying $\beta|_T$ on a zero measure subset)  that the restriction of $\beta$ to $T$ is a measurable homomorphism $\beta_T:M_x\to M_{\phi(x)}$ and hence coincides almost everywhere with an algebraic map. Since the action of $M_x$ and $M_{\phi(x)}$ on each vertical chain is algebraic and simply transitive, we  get that for almost every vertical chain $T$ the restriction of $\phi$ to $T$ is algebraic.
\end{proof}
The fact that we let $\beta$ depend on the vertical chain $T$ might be surprising, and it is probably possible to prove that the cocycle $\alpha$ coincides almost everywhere with an homomorphism that doesn't depend on the vertical chain $T$. However since it suffices to prove that the restriction of $\alpha$ to almost every vertical chain coincides almost everywhere with an homomorphism, and since this reduces the technicalities involved, we will restrict to this version. 

The rest of the section is devoted to prove, using the chain geometry of $\Ss_{m,n}$, that the hypothesis of Proposition \ref{prop:homom>algebraic} is satisfied whenever $\phi$ is a Zariski dense, chain geometry preserving map and $m<n$. In the following proposition we deal with a preliminary easy case, in which the geometric picture behind the general proof should be clear.

\begin{prop}\label{prop:easy}
Let $\phi:\deH^p\to \Ss_{m,n}$ be a measurable, Zariski dense, chain geometry preserving map, and let $n\geq 2m$. Then the restriction of $\phi$ to almost every chain coincides almost everywhere with a rational map.
\end{prop}
\begin{proof}
We want to apply Proposition \ref{prop:homom>algebraic} and show that the cocycle $\alpha:M_x\times \deH^p\to M_{\phi(x)}$ only depends on the vertical chain a point belongs to. We consider the set $\calF\subseteq W_x$ of chains $F$ such that 

\begin{minipage}{.7\textwidth}
\begin{enumerate}
 \item $\hat\phi(F)$ is an $m$-vertical chain, hence in particular $\phi_x(F)$ is defined, 
  \item for almost every circle $C$ containing the point $F\in W_x$, for almost every chain $T$ lifting $C$ the diagram  of Lemma \ref{lem:phix} commutes almost everywhere.
 \end{enumerate}
\end{minipage}
\hspace{.5cm}
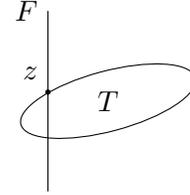
\begin{minipage}{.3\textwidth}
\vspace{-.5cm}
\begin{tikzpicture}[scale=.8]
   \draw (0,0) to (0,3);
   \node [left] at (0,3){$F$};

  \draw (1,1.5) circle [x radius=1.5, y radius=.5, rotate=15] ;
  \node at (1,1.5) {$T$};

  \node[above left] at (0,1.7) {$z$};
  \filldraw (0,1.65) circle (1pt);
\end{tikzpicture}
\end{minipage}

It follows from the proof of Lemma \ref{lem:phix} that the set $\calF$ is of full measure, moreover, we get, applying Fubini, that if $F$ is an element in $\calF$, for almost every point $z$ in $F$ and almost every chain $T$ through $z$ the diagram of Lemma \ref{lem:phix} commutes almost everywhere when restricted to $T$. In particular, using Fubini again, this implies that for almost every point $w$ in $\deH^p$ the diagram of Lemma \ref{lem:phix} commutes almost everywhere when restricted to the chain $T_{z,w}$.

\begin{minipage}{.56\textwidth}

Let us now fix a chain $F\in \calF$ and denote by $\Oo$ the full measure set of points in $F$ for which that holds. For every element $e\in M_x$ we also consider the full measure set $\Oo_{e}=\Oo\cap e^{-1}\Oo$.  
We claim that given two points $z_1,z_2\in \Oo_e$ the cocycle $\alpha(e,z_i)$ has the same value $\beta(e,F)$.
In fact let us  consider the set $\Aa_{z_1,z_2,e}\subseteq\deH^p$ consisting of points $w$  such that

 \begin{enumerate}
  \item  $\phi (w)\in\hat \phi(T_{w,z_1})\cap \hat \phi(T_{w,z_2})$
  \item $\phi (ew)\in\hat \phi(T_{ew,ez_1})\cap \hat \phi(T_{ew,ez_2})$
  \item $\dim \<\phi(z_1),\phi(z_2),\phi(w)\>=3m$.
  \end{enumerate}
 \end{minipage}
 \hspace{1.2cm}
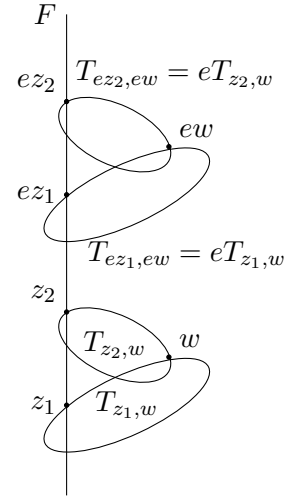
\begin{minipage}{.2\textwidth}

   \begin{tikzpicture}[scale=.8]
   \draw (0,0) to (0,8);
   \node [left] at (0,8){$F$};

  \draw (1,1.5) circle [x radius=1.5, y radius=.5, rotate=25] ;
  \node at (1,1.5) {$T_{z_1,w}$};

  \node[left] at (0,1.5) {$z_1$};
  \filldraw (0,1.5) circle (1pt);
  
  \draw (.8,2.5) circle [x radius=1, y radius=.5, rotate=-25] ;
  \node at (.8,2.5){$T_{z_2,w}$};
  
 \node[ above right] at (1.7,2.3) {$w$};
 \filldraw (1.7,2.3) circle (1pt);

 \node[ above left] at (0,3.05) {$z_2$};
 \filldraw (0,3.05) circle (1pt);

  \draw (1,5) circle [x radius=1.5, y radius=.5, rotate=25] ;
  \node at (2,4) {$T_{ez_1,ew}=eT_{z_1,w}$};

  \node[left] at (0,5) {$ez_1$};
  \filldraw (0,5) circle (1pt);
  
  \draw (.8,6) circle [x radius=1, y radius=.5, rotate=-25] ;
  \node at (1.8,7){$T_{ez_2,ew}=eT_{z_2,w}$};
  
 \node[ above right] at (1.7,5.8) {$ew$};
 \filldraw (1.7,5.8) circle (1pt);

 \node[ above left] at (0,6.55) {$ez_2$};
 \filldraw (0,6.55) circle (1pt);

  \end{tikzpicture}
\end{minipage}

 We claim that the set $\Aa_{z_1,z_2,e}$ is not empty. Indeed, by definition of  $\Oo_e$, the set of points $w$ satisfying the first two assumption is of full measure. Moreover, since $n\geq 2m$, the set $\mathcal C$ of points in $\Ss_{m,n}$ such that $\dim \<\phi(z_1),\phi(z_2),\phi(w)\>< 3m$ is a proper Zariski closed subset of $\Ss_{m,n}$. Since the map $\phi$ is Zariski dense, the preimage of $\mathcal C$ cannot have full measure, and this implies that $\Aa_{z_1,z_2,e}$ has positive measure, in particular it contains at least one point. The third assumption on the point $w$ implies that the $m$-chain containing $\phi(w)$ and $\phi(z_i)$ is horizontal for $i=1,2$.
 
 Let us fix a point $w\in \Aa_{z_1,z_2,e}$ and consider the $m$-chain $\hat \phi(eT_{w,z_i})$ for $i=1,2$. The $m$-chain $\hat \phi(eT_{w,z_i})$ is a lift of the $(m,0)$-circle $C_i=\pi_{\phi(x)}(\hat\phi(T_{w,z_i}))$ that contains both the points $\phi(e z_i)$ and $\phi(ew)$. In particular, since $\alpha(e,z_i)\hat \phi (T_{w,z_i})$ is a lift of $C_i$ containing $\phi(ez_i)$ we get that $ \alpha(e,z_i)\hat \phi (T_{w,z_i})=\hat \phi(eT_{w,z_i})$. Similarly we get that $ \alpha(e,w)\hat \phi (T_{w,z_i})=\hat \phi(eT_{w,z_i})$. This gives that $\alpha(e,z_i)^{-1} \alpha(e,w)\in M_{\hat\phi(T_{w,z_i})}$, but the latter group is the trivial group since we know that the chain $\hat\phi(T_{w,z_i})$ is 0-vertical. This implies that $\alpha(e,z_1)=\alpha(e,w)=\alpha(e,z_2)$.
\end{proof}
\subsection{Proof of Theorem \ref{thm:restriction rational}}
Let us now go back to the setting of Theorem \ref{thm:restriction rational}: we fix a measurable, chain geometry preserving, Zariski dense map $\phi:\deH^p\to \Ss_{m,n}$, a generic point $x\in\deH^p$ such that for almost every chain $t\in W_x$, for almost every point $y\in t$, $\phi(y)\in \hat \phi(t)$. We want to show that the measurable cocycle $\alpha:\deH^p\backslash\{x\}\times \mathfrak u(1)\to\mathfrak u(m)$ coincides on almost every vertical chain with a measurable homomorphism. In particular it is enough to show that for almost every pair $z_1,z_2$ on a vertical chain in $\deH^p$, the values  $\alpha(e,z_1)$ and $\alpha(e,z_2)$ coincide. For a generic pair $(z_1,z_2)$ we have that the triple $(\phi(x),\phi(z_1),\phi(z_2))$ is contained in a tube type subdomain and hence we can compose the map $\phi$ with an element of the group $\SU(m,n)$ so that $\phi(x)=v_\infty$, $\phi(z_1)=v_0$ and $\phi(z_2)=v_d$, here and in the following we denote by $v_d$ the subspace with basis $\bpm\Id&0&d\epm^T$ for some 
diagonal matrix $d$ with all entries equal to $\pm i$. In fact it is proven in \cite[Theorem 5.2]{
CN} that the Bergmann cocycle is a complete invariant for the $\SU(m,n)$ action on triples of pairwise transverse points in an $m$-chain, and varying the matrix $d$ one gets that  $\beta_{\Ss}(v_0,v_\infty, v_d)$ achieves all possible values     (cfr. Proposition \ref{prop:bergmann}). 

For the rest of the section we restrict to the case $n<2m$, since the otherwise Theorem \ref{thm:restriction rational} follows from Proposition \ref{prop:easy} and denote by $l$ the integer $l=n-m$ and $k=2m-n$. In analogy with the proof of Proposition \ref{prop:easy} we denote by $\Aa_{z_1,z_2,e}$ to be the full measure subset of $\deH^p$ consisting of points with 
\begin{enumerate}
   \item $\phi (w)\in\hat \phi(T_{w,z_1})\cap \hat \phi(T_{w,z_2}), $
   \item $\phi (ew)\in\hat \phi(T_{ew,ez_1})\cap \hat \phi(T_{ew,ez_2}), $
   \item$\dim \<v_0,v_d,\;\phi(w)\>=m+n.$
  \end{enumerate}

We will consider the subset $\Dd_{v_0,v_d}$ of $\Ss_{m,n}$ defined by 
$$\Dd_{v_0,v_d}=\{w\in\Ss_{m,n}|\; w \text{ is transverse to } v_0,v_d\text{ and } \<v_0,v_d\> \}.$$
It is easy to verify that $\Dd_{v_0,v_d}$ consists of points $w$ such that both chains $T_{v_0,w}$ and $T_{v_d,w}$ are well defined and $k$-vertical: indeed in our assumptions $w$ is transverse to $\<v_0,v_d\>=\<v_\infty,v_0\>$. In particular $\dim\<w,v_0,v_\infty\>=n+m$ and we get
$$n+m=\dim\<w,v_0\>+\dim v_\infty-\dim (v_\infty\cap \<v_0,w\>)$$
which implies that $i_{v_\infty}(T_{v_0,w})=3m-(n+m)=k$.

Our next goal is to associate to any point $w$ in $\Dd_{v_0,v_d}$ subgroups $E(w)$, $I(w)$ which, when we consider a point $w$ that is image of a point $z$ in $\Aa_{v_0,v_d,e}$, represent, respectively, the error allowed by the point $z$ for the cocycle $\alpha(e,v_0)$ and some information on the difference $\alpha(e,v_0)-\alpha(e,v_d)$ obtained applying the strategy of Proposition \ref{prop:easy} to the point $z$.

In order to define the subgroups properly, we use the identification $M_{v_\infty}=\fru(m)$ provided in Section \ref{sec:S_mn} and use the linear structure on $\fru(m)$. Moreover we will denote by $H $ the positive definite bilinear form on $\fru(m)$ given by $H(A,B)=\tr A^*B$.  We also denote by $\beta_{v_0} $ the map
$$\begin{array}{cccc}\beta_{v_0}:&\Dd_{v_0,v_d}\subseteq\Ss_{m,n}&\to&\Gr_k(v_\infty)\\&w&\to&\<v_0,w\>\cap v_\infty\end{array}$$
similarly we define $\beta_{v_d}$ so that $\beta_{v_d}(w)=\<v_d,w\>\cap v_\infty$.

We want to understand the subspaces on which the possible defect of the cocycle $\alpha$ to be an homomorphism are confined. Whenever two $k$-dimensional subspaces $Z_i$ of $\C^m$ are fixed we denote by $S(Z_1,Z_2)$ the subspace:
$$S(Z_1,Z_2)=\<z_1z_2^*-z_2z_1^*|\;z_i\in Z_i\><\mathfrak u(m).$$

The map $S$ is useful to define the error subgroup $E(w)$ associated to a point $w$ in $\Dd_{v_0,v_d}$:
$$E(w)=S(\beta_{v_0}(w),\beta_{v_0}(w))+S(\beta_{v_d}(w)\beta_{v_d}(w)).$$
%

For each point $z$ in the set $\Aa_{z_1,z_2,e}$ the error group $E(\phi(z))$ bounds the error of the cocycle $\alpha$:
  \begin{lem}
For every point $z$ in $\Aa_{z_1,z_2,e}$ we get $\alpha(e,z_1)-\alpha(e,z_2)\in E(\phi(z)).$ 
\end{lem}
\begin{proof}By the assumption on $z$ we have that  $\hat \phi(T_{ez,ez_1})$ and $\hat \phi(T_{z,z_1})$ project to the same $(m,k)$-circle, and in particular we get that 
 $$\alpha(e,z)-\alpha(e,z_1)\in M_{\hat\phi( T_{z,z_1})}.$$
 In the same way one gets that 
$$\alpha(e,z)-\alpha(e,z_2)\in M_{\hat\phi (T_{z,z_2})}.$$
It follows from Proposition \ref{prop:errors} that if $g_i\in\GL(m)$ is such that $g_1\beta_{v_0}(\phi(w))=\<e_1,\ldots, e_k\>$  (resp. $g_2\beta_{v_d}(\phi(w))=\<e_1,\ldots, e_k\>$) we have that  $M_{\hat\phi (T_{w,z_i})}=i(g_i^{-1}E_kg_i^{-*})$, and this proves our first claim since it is easy to check, from the definition of the set $E(\phi(w))$ that $$E(\phi(w))=g_1^{-1}E_kg_1^{-*}+g_2^{-1}E_kg_2^{-*}.$$
\end{proof}

 In particular it is enough to show that, for almost every pair $z_1,z_2$ on a vertical chain, the intersection $\bigcap_{z\in \Aa_{z_1,z_2,e}}E(\phi(z))=\{0\}$. In fact this would imply that the restriction of $\alpha$ to almost every chain essentially doesn't depend on  the choice of the point, hence coincides with a measurable homomorphism. In order to do this we define another subgroup of $M_{v_\infty}$ associated to a point $w\in \Dd_{v_0,v_d}$. The information associated to $w$ will be
 $$I(w)=S\left((\beta_{v_0}(w)^\bot, \beta_{v_d}(w)^\bot\right)$$
 Here the orthogonals are considered with respect to the standard Hermitian form on $v_\infty=\C^m$. It is easy to verify that $I(w)$ is contained in the orthogonal to $E(w)$ with respect to the orthogonal form on $\fru(m)$ given by $H(A,B)=\tr(A^*B)$.
 
 We postpone the proof of the following technical lemma to the next section:
 \begin{lem}\label{cor:aeE(z)cuts}
 For every proper subspace $L$ of $\mathfrak u (m)$ the set $C(L)=\{z\in \Dd_{v_0,v_d}^{v_\infty}|\;I(z)\subseteq L\}$ is a proper Zariski closed subset of $\Dd_{v_0,v_d}^{v_\infty}\subseteq\Ss_{m,n}$.
\end{lem}
 We now conclude the proof of Theorem \ref{thm:restriction rational} assuming Lemma \ref{cor:aeE(z)cuts}.
 \begin{proof}[Proof of Theorem \ref{thm:restriction rational} ]
We choose $m^2$ points $w_1,\ldots,w_{m^2}$ in $\deH^p$ such that $\<I(\phi(w))\>=\fru(m)$: we work by induction and assume that there exist $j$ points $w_1,\ldots, w_j$ with $\dim L_j=\dim\<I(\phi(w_i))|\;i\leq j\>\geq j$. If the set  $L_j$
is equal to the whole $\mathfrak u(m)$ we are done. Otherwise it follows from Lemma \ref{cor:aeE(z)cuts} that the subset $C(L_j)$ of $\Dd_{v_0,v_d}^{v_\infty}$  is a proper Zariski closed subset of $\Dd_{v_0,v_d}^{v_\infty}$. 

In particular, since $\phi$ is Zariski dense, its essential image cannot be contained in $C(L_j)\cup (\Ss_{m,n}\setminus\Dd_{v_0,v_d}^{v_\infty})$ that is a Zariski closed subset of $\Ss_{m,n}$. Hence we can find a point $w_{j+1}$ in the full measure set $\Aa_{z_1,z_2,e}$ such that $I(\phi(w_{j+1}))$ is not contained in $L_j$, and this implies that $L_{j+1}=\< I(\phi(w_i))|\;i\leq j+1\>$ strictly contains  $L_j$, hence has dimension strictly bigger than $j$. This completes the proof of Theorem \ref{thm:restriction rational}.
\end{proof}

\subsection{Possible errors}\label{ssec:4.3}
 
A crucial step in the proof of Lemma \ref{cor:aeE(z)cuts} is to show that the map $\beta=\beta_{v_0}\times\beta_{v_d}:\Dd_{v_0,v_d}\subseteq\Ss_{m,n}\to\Gr_k(v_\infty)^2$ is surjective (cfr. Proposition \ref{prop:betasurj}). As a preparation for this result we give a parametrization of the image of $\Dd_{v_0,v_d}$ under the map $\pi_{v_0}\times\pi_{v_d}:\Dd_{v_0,v_d}\to W_{v_0}\times W_{v_d}$. It is easy to check that explicit parametrizations of $W_{v_0}$ and $W_{v_d}$ are given by 
 $$\begin{array}{ccc}\begin{array}{ccc}M(l\times m,\C)&\to&W_{v_0}\\A_0&\to&\bpm 0\\ \Id_l\\ A_0^*\epm^\bot\end{array}&&\begin{array}{ccc}M(l\times m,\C)&\to&W_{v_d}\\A_1&\to&\bpm A_1^*\\ \Id_l\\ dA_1^*\epm^\bot\end{array}\end{array}$$
 Moreover a point $A_i$ in $W_{v_i}$ correspond to a $k$-vertical chain if ${\rm rk}(A_i)=m-k$. 
 This allows us to give an explicit description of the image:
 \begin{lem}\label{lem:4.15}
Under the parametrizations above, the image of  the map $\pi_{v_0}\times\pi_{v_d}:\Dd_{v_0,v_d}^{v_\infty}\to W_{v_0}\times W_{v_d}$ is the closed subset $\Cc_{v_0,v_d}$ of $M(l\times m,\C)\times M(l\times m,\C)$ defined by  
$$\Cc_{v_0,v_d}=\left\{(A_0,A_1)\left|\begin{array}{l}A_1A_0^*\in\Id+\U(m)\\
                             A_0,A_1 \text{ have maximal rank}
                            \end{array}
\right.\right\}.$$
\end{lem}

\begin{proof}
We already observed that each point $w$ in $\Dd_{v_0,v_d}$ uniquely determines two $k$-vertical chains $T_{w,v_0}$, $T_{w,v_d}$. In particular for each pair $(A_0,A_1)$ in the image of $\pi_{v_0}\times\pi_{v_d}$ we have that $A_i$ has maximal rank. We will now show that $A_1A_0^*\in\Id+\U(m)$ if an only if the intersection of linear subspaces associated to the two $m$-chains contains a maximal isotropic subspace.

 Two $m$-chains $T_0$, $T_1$ intersect in $\Ss_{m,n}$ if and only if the intersection $V_0\cap V_1$ of their underlying vector spaces $V_0$, $V_1$ contains a maximal isotropic subspace. In turn this is equivalent to the requirement that $(V_0\cap V_1)^\bot$ has signature $(0,l)$. Indeed, since $V_0^\bot$ has signature $(0,l)$ and is contained in  $(V_0\cap V_1)^\bot$, we get that the signature of any subspace of $(V_0\cap V_1)^\bot$ is $(k_1,l+k_2)$ for some $k_1, k_2$. On the other hand if $V_0\cap V_1$ contains a maximal isotropic subspace  $z$, then $(V_0\cap V_1)^\bot\subseteq z^\bot$ and the latter space has signature $(0,l)$. In particular the signature of $(V_0\cap V_1)^\bot$ would be $(0,l)$, and clearly the orthogonal of a subspace of signature $(0,l)$ contains a maximal isotropic subspace.

 Since $(V_0\cap V_1)^\bot=\< V_0^\bot, V_1^\bot\>$, we are left to check that the requirement that signature of this latter subspace is $(0,l)$ is equivalent to the requirement that $A_1A_0^*$ belongs to $\Id+ U(l)$.
 If now we pick a pair $(A_0,A_1)\in M(l\times m,\C)\times M(l\times m,\C)$ representing a pair of subspaces $(V_0,V_1)\in W_{v_0}\times W_{v_d}$ we have that the subspace $(V_0\cap V_1)^\bot$ is spanned by the columns of the matrix
 $$\bpm0&A_1^*\\\Id&\Id\\A_0^*&dA_1^*\epm.$$
 It is easy to compute the restriction of $h$ to the given generating system of $(V_0\cap V_1)^\bot$:
 $$\begin{array}{c}\bpm0&\Id&A_0\\A_1&\Id&-A_1d\epm\bpm&&\Id\\&-\Id&\\\Id&&\epm\bpm0&A_1^*\\\Id&\Id\\A_0^*&dA_1^*\epm=\\
 \bpm0&\Id&A_0\\A_1&\Id&-A_1d\epm\bpm A_0^*&dA_1^*\\-\Id&-\Id\\0&A_1^*\epm=\bpm-\Id&A_0A_1^*-\Id\\A_1A_0^*-\Id&-\Id\epm
 \end{array}
 $$
 The latter matrix is negative semidefinite and has rank $l$ if and only if \begin{equation}\label{eqn:1}
 A_0A_1^*A_1A_0^*-A_1A_0^*-A_0A_1^*=(A_0A_1^*-\Id)(A_1A_0^*-\Id)-\Id=0.
 \end{equation} 
 In this case the restriction of $ h$ to $(V_0\cap V_1)^\bot$ has signature $(0,l)$.
 
 The intersection $V_0\cap V_1$ contains maximal isotropic subspaces that are transverse to $v_0$ and $v_d$ if and only if the radical of $V_0\cap V_1$, which coincides with the radical of $(V_0\cap V_1)^\bot$, is transverse to both subspaces. It is easy to verify that this is always the case if $A_0$ and $A_1$ have maximal rank.
 \end{proof}
 
We now turn to the analysis of the map 
$$\begin{array}{cccc}\beta:&\Dd_{v_0,v_d}^{v_\infty}\subseteq\Ss_{m,n}&\to&\Gr_k(v_\infty)^2\\&z&\to&(\<v_0,z\>\cap v_\infty,\<v_d,z\>\cap v_\infty).\end{array}$$

\begin{prop}\label{prop:betasurj}
The map $\beta$ is surjective.
\end{prop}

\begin{proof}
 If we denote by $\zeta$ the uniquely defined map with the property that  $\beta=\zeta\circ(\pi_{v_0}\times\pi_{v_d})$, then it is easy to check that the map $\zeta$ has the following expression, with respect to the coordinates described above:
 
 $$\begin{array}{cccc}\zeta:&W_{v_0}\times W_{v_d}&\to&\Gr_k(v_\infty)^2\\ &(A_0,A_1)&\mapsto &(\ker(A_0), \ker(A_1)).\end{array}$$

 In order to conclude the proof it is enough to show that any pair $(V_0,V_1)$ of $k$-dimensional subspaces of $v_\infty$ can be realized as the kernels of a pair of matrices satisfying Equation \ref{eqn:1}.
 
We first consider the case in which the subspaces $V_0,V_1$ intersect trivially, of course this can only happen if $k\leq l$.
In this case there exists an element $g\in \U(m)$ such that  $gV_0=\wt V_0=\<e_1\ldots,e_k\>$ and that $\wt V_1=gV_1$ is spanned by the columns of the matrix $\bsm B\\\Id_k\\0\esm$ where $B$ is a matrix in $M(k\times k,\C)$. Clearly $V_i$  is the kernel of $A_i$ if and only if $\wt V_i$ is the kernel of $\wt A_i=A_i g^{-1}$ and $A_ig^{-1}$ satisfies the equation \ref{eqn:1} if and only if $A_i$ does. In particular it is enough to exhibit matrices $\wt A_i$ whose kernel is $\wt V_i$.

Let us first notice that, for any matrix $B\in M(k\times k,\R)$ there exists a matrix $X\in \GL_k(\C)$ such that $XB$ is a diagonal matrix $D$ whose elements are only 0 or 1. Let us now consider the matrices 
$$\begin{array}{cc}
\begin{array}{rcl}
\wt A_1^*=&\left[\begin{array}{ccc}0&2\Id&0\\0&0&2\Id\end{array}\right]&\hspace{-10pt}\begin{array}{l}_k\\_{l-k}\end{array}\\
\end{array}
&
\begin{array}{rcl}
\wt A_2^*=&\left[\begin{array}{ccc}X&D&0\\0&0&\Id\end{array}\right]&\hspace{-10pt}\begin{array}{l}_k\\_{l-k}\end{array}\\
\end{array}
\end{array}
$$

By construction $\wt V_1$ is the kernel of $\wt A_1$ and $\wt V_2$ is the kernel of $\wt A_2$, moreover we have that $\wt A_1\wt A_2^*$ satisfies Equation \ref{eqn:1}:
$$\wt A_1^*\wt A_2^=\bpm0&2\Id&0\\0&0&2\Id\epm\bpm X^*&0\\D^*&0\\0&\Id\epm=\bpm2D^*&0\\0&2\Id\epm\in \Id+U(m).$$ This implies that there is a point $z\in\Ss_{m,n}$ with $\beta(z)=(\wt V_0,\wt V_1)$.

The general case, in which the intersection of $V_i$ is not trivial, is analogous: we can assume, up to the $\U(m)$ action that $V_0\cap V_1=\<e_1,\ldots,e_s\>$ and we can restrict to the orthogonal to  $V_0\cap V_1$ with respect to the standard Hermitian form.
\end{proof}
 
We can now prove Lemma \ref{cor:aeE(z)cuts}.
\begin{proof}[Proof of Lemma \ref{cor:aeE(z)cuts}]
 The subspace $C(L)$ is Zariski closed since the subspaces of $\mathfrak u(m)$ that are contained in $L$ form a Zariski closed subset of the Grassmanian $\Gr(\fru(m))$ of the vector subspaces of $\fru(m)$, moreover the subspace $I(z)$ is obtained as the composition $I(z)=S\circ \beta$ of two regular maps.
 
 In order to verify that $C(L)$ is a proper subset, unless $L=\mathfrak u(m)$, it is enough to verify that the subspaces of the form $S(Z_1^\bot,Z_2^\bot)$ with $Z_i$ transverse subspaces span the whole $\mathfrak u(m)$. Once this is proven, the result follows from the the surjectivity of $\beta$: since $\beta$ is surjective, the preimage of a proper Zariski closed subset is a proper Zariski closed subset. The fact that the span
 $$\<z_1z_2^*-z_2z_1^*|\;z_1,z_2\in \C^{m} \text{ linearly independent}\>$$
 is the whole $\mathfrak u(m)$ follows from the fact that every matrix of the form $iz_1z_1^*$ is in the span, since such a matrix can be obtained as the difference $z_1(z_2-iz_1)^*-(z_2-iz_1)z_1^*-(z_1z_2^*-z_2z_1^*)$.
\end{proof}

\section{The boundary map is rational}\label{sec:reduction}
In this section we will show that a Zariski dense, chain geometry preserving map $\phi:\deH^p\to\Ss_{m,n}$ whose restriction to almost every chain is rational, coincides almost everywhere with a rational map.

Assume that the chain geometry preserving map $\phi$ is rational and let us fix a point $x$. Since the projection $\pi_{\phi(x)}:\Ss_{m,n}^{\phi(x)}\to W_{\phi(x)}$ is regular we get that the map $\phi_x:W_x\to W_{\phi(x)}$ induced by $\phi$ is rational as well. The first result of the section is that the converse holds, namely that if there exist enough many generic points $s_1,\ldots,s_l$ in $\deH^p$ such that $\phi_{s_i}$ is rational, then the original map $\phi$ had to be rational as well.

In what follows we will denote by $l$ the smallest integer bigger than $1+m/(n-m)$.
 \begin{lem}\label{lem:Oo}
  There exists a Zariski open subset $\Oo\subset \Ss_{m,n}^l$, such that for any $(x_1,\ldots,x_l)$ in $\Oo$, there exists a Zariski open subset $\Dd_{x_1,\ldots,x_l}\subset \Ss_{m,n}$ such that for every $z\in \Dd_{x_1,\ldots,x_l}$ we have
  $$\bigcap_{i=1}^l\<z,x_i\>=z.$$
 \end{lem}
\begin{proof}
 Let us consider the set $\Ff$ of $(l+1)$-tuples $(x_1,\ldots,x_l,z)$ in $\Ss_{m,n}^{l+1}$ with the property that $\bigcap_{i=1}^l\<z,x_i\>=z$. This is a Zariski open subset of $\Ss_{m,n}^{l+1}$: indeed, since $z$ is clearly contained in the intersection, the set $\Ff$ is defined by the equation $\dim\bigcap_{i=1}^l\<z,x_i\>\leq m$. In order to conclude the proof it is enough to show that for each $z\in \Ss_{m,n}$ the set of tuples $(x_1,\ldots,x_l)$ with the property that $(x_1,\ldots,x_l,z)\in\Ff$ is non empty: this implies that the set $\Ff$ is a non empty Zariski open subset, and in particular there must exist a Zariski open subset of $\Ss_{m,n}^l$ consisting of $l$-tuples $(x_1,\ldots,x_l)$ satisfying the hypothesis of the lemma. 
 
 Let us then fix a point $z\in \Ss_{m,n}$. We denote by $\Aa_{z}^k$ the set of $k$-tuples $x=(x_1,\ldots,x_k)$ in $\Ss_{m,n}^z$ such that $\dim\bigcap_{i=1}^k\<z,x_i\>=\max\{2m-(n-m)(k-1),m\}$. In order to conclude the proof it is enough to exhibit, for every $k$-tupla $x$ in $\Aa_z^k$ a non empty subset $\Bb$ of $\Ss_{m,n}$ such that $(x,b)\in\Aa_z^{k+1}$ for each $b$ in $\Bb$. If we denote by $V_k$ the subspace $\bigcap_{i=1}^k\<z,x_i\>$, that has, by our assumption on the tuple $x$, dimension $2m-(n-m)(k-1)$, we can take $\Bb$ to be the Zariski open subset 
 $$\Bb=\{x_{k+1}\in\Ss_{m,n}|x_{k+1}\tra V_k, x_{k+1}\tra z\}.$$
 The set $\Bb$ is not empty since both transversality conditions are non-empty, Zariski open conditions, and for this choice we get
 $$\begin{array}{rl}
 \dim\bigcap_{i=1}^{k+1}\<z,x_i\>&=\dim(V_k\cap \<z,x_{k+1}\>)\\
 &=\dim V_k+\dim \<z,x_{k+1}\>-\dim\<V_k,x_{k+1}\>=\\
 &=2m-(n-m)(k-1)+2m\\&-\min\{m+n,\,2m-(n-m)(k-1)+m\}=\\
 &=\max\{2m-(n-m)k,m\}.
 \end{array}$$
\end{proof}

It is worth remarking that, if $n\geq 2m$, the set $\Ss_{m,n}^{(2)}$ is contained in $\Oo$ and for $x_1,x_2$ transverse the set $\Dd_{x_1,x_2}$ consists of the points $z$ that are transverse to $x_1,x_2$ and $\<x_1,x_2\>$. This is consistent with the notation in Section \ref{ssec:4.3}. In general we will assume (up to restricting $\Dd_{x_1,\ldots,x_l}$ to a smaller Zariski open subset) that each $z$ in $\Dd_{x_1,\ldots,x_l}$ is transverse to $x_i$ for each $i$.

\begin{lem}
Let $(x_1,\ldots,x_l)$ be an $l$-tuple of pairwise transverse points in the set $\Oo$ defined in Lemma \ref{lem:Oo}. 
There exist a quasiprojective subset $\Cc_{x_1,\ldots,x_l}$ of $W_{x_1}\times\ldots\times W_{x_l}$ such that the map $\beta_{x_1,\ldots,x_l}=\pi_{x_1}\times\ldots\times\pi_{x_l}:\Dd_{x_1,\ldots,x_l}\to W_{x_1}\times\ldots\times W_{x_l}$ gives a birational isomorphism.
\end{lem}
\begin{proof}
We consider the set $\Cc'_{x_1,\ldots,x_l}$ consisting of tuples $(t_1,\ldots,t_l)$ with the property that the associated linear subspaces intersect in an $m$-dimensional isotropic subspace, and that $t_j-\pi_j(x_i)$ has maximal rank for every $i,j$. With this choice $\Cc'_{x_1,\ldots,x_l}$ is quasiprojective since the condition that the intersection has dimension at least $m$ and that the restriction of $h$ to the intersection is degenerate are closed condition (defined by polynomial), the condition that the intersection has dimension at most $m$ is an open condition. The set $\Cc_{x_1,\ldots,x_l}$ is the subset of $\Cc'_{x_1,\ldots,x_l}$ that is the image of $\beta_{x_1,\ldots,x_l}$.

The fact that the map $\beta_{x_1,\ldots,x_l}$ gives a birational isomorphism follows from the fact that a regular inverse to $\beta_{x_1,\ldots,x_l}$ is given by the algebraic map that associates to an $l$-tuple of points their unique intersection. 
\end{proof}
We now have all the ingredients we need to prove the following
\begin{prop}
 Let us assume that for almost every point $x\in\deH^p$ 
 the map $\phi_x$ coincides almost everywhere with a rational map. The same is true for $\phi$.
\end{prop}
\begin{proof}
Let us fix $l$ points $t_1,\ldots, t_l$ which are generic in the sense of Lemma \ref{lem:1}, which satisfy that $\phi_{t_i}$ coincides almost everywhere with a rational map, and with the additional property that the $l$-tupla $(\phi(t_1),\ldots,\phi(t_l))$ belongs to the set $\Oo$. We can find such points since the map $\phi$ is Zariski dense and the set $\Oo$ is Zariski open. Let us now 
consider the diagram
$$\xymatrix{\deH^p\setminus\{t_1,\ldots,t_l\}\ar[r]^{\phi}\ar[d]_{\pi_{t_1}\times\ldots\times\pi_{t_l}}&\Dd_{\phi(t_1),\ldots,\phi(t_l)}\subseteq \Ss_{m,n}\ar[d]^{\beta_{\phi(t_1),\ldots,\phi(t_l)}}\\ W_{t_1}\times\ldots\times W_{t_n}\ar[r]^-{\phi_{t_1}\times\ldots\times\phi_{t_n}}&\Cc_{\phi(t_1),\ldots,\phi(t_l)}}.$$
A consequence of Lemma \ref{lem:phix} and of the definition of the isomorphisms $\beta_{\phi(t_1),\ldots,\phi(t_l)}$ is that the diagram commutes almost everywhere. In particular, since the isomorphism $\beta_{\phi(t_1),\ldots\phi(t_l)}$ is birational, and $\pi_{t_1}\times\ldots\times\pi_{t_l}$ is rational, we get that $\phi$ coincides almost everywhere with a rational map.
\end{proof}
Let us now fix a point $x$ in $\deH^p$, and identify the space $W_x$ with $\C^{p-1}$. We want to study the map $\phi_x:\C^{p-1}\to W_{\phi(x)}$. 
It follows from Lemma \ref{lem:parametrization of k-vertical chains} restricted to the case $m=1$ that the projections of chains in $\deH^p$ to $\C^{p-1}$ are Euclidean circles $C\subset \C^{p-1}$ (possibly collapsed to points). 

\begin{lem}
 If $x$ is generic in the sense of Lemma \ref{lem:1}, the restriction of $\phi_x$ to almost every Euclidean circle $C$ of $\C^{p-1}$ is rational. 
\end{lem}

\begin{proof}

 It follows from the explicit parametrization of a chain given in Lemma \ref{lem:parametrization of k-vertical chains}  that, whenever a  point $t$ in $\pi_x^{-1}(C)$ is fixed, the lift map $l:C\to T$ is algebraic, where $T$ is the unique lift of $C$ containing $t$.
 
 In particular, if $T$ is a chain such that the restriction of $\phi$ to $T$ coincides almost everywhere with a rational map, the restriction of $\phi_x$ to $C=\pi_x(T)$ coincides almost everywhere with a rational map. We can now use a Fubini based argument to get that, for almost every circle $C$, the restriction of $\phi_x$ to $C$ is rational: for almost every chain $T$, the restriction to $T$ coincides almost everywhere with a rational map, the conclusion follows from the fact that the space of chains that do not contain $x$ is a full measure subset of the space of chains in $\deH^p$ that forms a smooth bundle over the space of Euclidean circles in $\C^{p-1}$.
\end{proof}
 
 An usual Fubini type argument implies now the following
 \begin{cor}\label{cor:phiL}
  For almost every complex affine line $\Ll\subseteq \C^{m}$, for almost every Euclidean circle $C$ contained in $\Ll$, the restriction of $\phi_x$ to $C$ is algebraic. Moreover the same is true for almost every point $p$ in $\Ll$ and almost every circle $C$ containing $p$. 
 \end{cor}
In order to conclude the proof we will apply many times the following well known lemma that allows to deduce that a map is rational provided that the restriction to sufficiently many subvarieties is rational. Given a map $\phi:A\times B\to C$ and given a point $a\in A$ we denote by $_a\phi:B\to C$ the map $_a\phi(b)=\phi(a,b)$ in the same way, if $b$ is a point in $B$, $^b\phi$ will the note the map $^b\phi(a)=\phi(a,b)$ 

\begin{lem}[{\cite[Theorem 3.4.4]{Zimmer}}]\label{lem:R^nxR^m}
 Let $\phi:\R^{n+m}\to\R$ be a measurable function. Let us consider the splitting $\R^{n+m}=\R^n\times\R^m$. Assume that for almost every $a\in \R^n$ the function $_a\phi:\R^m\to\R$ coincides almost everywhere with a rational function and for almost every $b\in\R^m$ the function $^b\phi:\R^n\to\R$ coincides almost everywhere with a rational function, then $\phi$ coincides almost everywhere with a rational function.
\end{lem}

 This easily gives that the restriction of $\phi_x$ to any complex affine line $\Ll$ in $\C^{p-1}$ coincides almost everywhere with a rational map:
\begin{lem}
 For almost every affine complex line $\Ll\subset \C^{p-1}$, the restriction $\phi_x|_{\Ll}$ coincides almost everywhere with a rational map.
\end{lem}
\begin{proof}
 Let us fix a line $\Ll$ satisfying the hypothesis of Corollary \ref{cor:phiL} and denote by $\phi_\Ll:\C\to W_{\phi(x)}$ the restriction of $\phi_x$ to $\Ll$, composed with a linear identification of $\Ll$ with the complex plane $\C$. By the second assertion of Corollary \ref{cor:phiL}, we can find a point $p\in \C$ such that for almost every Euclidean circle $C$ through $p$ the restriction of $\phi_{\Ll}$ to $C$ coincides almost everywhere with a rational map. Let us now consider the birational map $i_p:\C\to\C$ defined by $i_p(z)=(z-p)^{-1}$, and let us denote by $\psi_\Ll$ the composition $\psi_\Ll=\phi_\Ll\circ i_p^{-1}$. Since the image under $i_p$ of Euclidean circles through the point $p$ are precisely the affine real lines of $\C$ that do not contain $0$, we get that the restriction of $\psi_\Ll$ to almost every affine line coincides almost everywhere with a rational map. In particular a consequence of Lemma \ref{lem:R^nxR^m} is that the map $\psi_\Ll$ itself coincides almost everywhere with a 
rational map. Since $\phi_
\Ll$ coincides almost everywhere with $\psi_\Ll\circ i_p$ we get that the same is true for the map $\phi_\Ll$ and this concludes the proof.    
\end{proof}

Applying Lemma \ref{lem:R^nxR^m} again we deduce the following proposition:
\begin{prop}
 Let $\phi:\deH^p\to\Ss_{m,n}$ be a map with the property that for almost every chain $C$ the restriction of $\phi$ to $C$ coincides almost everywhere with a rational map. Then for almost every point $x\in\deH^p$ the map $\phi_x$ coincides almost everywhere with a rational map.
 \end{prop}
In turn this was the last missing ingredient to prove Theorem \ref{thm:phirational}

\section{Conclusion}\label{sec:rational}
The last step of Margulis' original proof of superrigidity involves showing  that if a Zariski dense representation $\rho:\G\to H$ of a lattice $\G$ in the algebraic group $G$ admits an algebraic boundary map, then it extends to a representation of $G$ (cfr. \cite{Margulis} and \cite[Lemma 5.1.3]{Zimmer}). The same argument applies here to deduce our main theorem: 

\begin{proof}[Proof of Theorem \ref{thm:Zariskisuperrigidity}]
Let $\rho:\G\to \PU(m,n)$ be a Zariski dense maximal representation and let $\psi:\deH^p\to\Ss_{m,n}$ be a measurable $\rho$-equivariant boundary map, that exists as a consequence of Proposition \ref{prop:boundary map} (the difference between $\SU(m,n)$ and $\PU(m,n)$ plays no role here, since the action of $\SU(m,n)$ on $\Ss_{m,n}$ factors through the projection to the adjoint form of the latter group). The essential image of $\psi$ is a Zariski dense subset of $\Ss_{m,n}$ as a consequence of Proposition \ref{prop:phiZariskidense}, moreover Corollary \ref{cor:incidence preserved} implies that $\psi$ preserves the chain geometry.

Since we proved that any measurable, Zariski dense, chain preserving boundary map $\psi$ coincides almost everywhere with a rational map (cfr. Theorem \ref{thm:phirational}), we get that there exists a $\rho$-equivariant rational map $\phi:\deH^p\to\Ss_{m,n}$. The $\rho$-equivariance follows from the fact that  $\phi$ coincides almost everywhere with $\psi$ that is $\rho$-equivariant. In particular, for every $\gamma$ in $\G$ the set on which $ \phi(\gamma x)=\rho(\gamma)\phi(x)$ is a Zariski closed, full measure set, and hence is the whole $\deH^p$.

Since $\phi$ is $\rho$-equivariant and rational it is actually regular: indeed the set of regular points for $\phi$ is a non-empty, Zariski open, $\G$-equivariant subset of $\deH^p$. Since, by Borel density \cite[Theorem 3.2.5]{Zimmer}, the lattice $\G$ is Zariski dense in $\SU(1,p)$ and $\deH^p$ is an homogeneous algebraic $\SU(1,p)$ space, the only $\G$-invariant proper Zariski closed subset of $\deH^p$ is the empty set, and this implies that the set of regular points of $\phi$ is the whole $\deH^p$.

In the sequel it will be useful to deal with complex algebraic groups and complex varieties in order to exploit algebraic results based on Nullstellensatz. This is easily achieved by considering the complexification. We will denote by $G$ the algebraic group $\SL(p+1,\C)$ and by $H$ the group $\PSL(m+n,\C)$ endowed with the appropriate real structures so that  $\SU(1,p)=G(\R)$ and $\PU(m,n)=H(\R)$.
Since $\deH^p$ and $\Ss_{m,n}$ are homogeneous spaces that are projective varieties, there exist parabolic subgroups $P<\SL(p+1,\C)$ and $Q<\PSL(m+n,\C)$ such that $\deH^p=(G/P)(\R)=G(\R)/(P\cap G(\R))$ and $\Ss_{m,n}=(H/Q)(\R)$. 

The algebraic $\rho$-equivariant map $\phi:\deH^p\to\Ss_{m,n}$ lifts to a map $\ov \phi:G(\R)\to \Ss_{m,n}$ and we can extend the latter map uniquely to an algebraic map $T:G\to H/Q$ using the fact that $G(\R)$ is Zariski dense in $G$. The extended map $T$ is $\rho$-equivariant since $G(\R)$ is Zariski dense in $G$: whenever an element $\gamma \in \G$ is fixed, the set $\{g\in G|\; T(\g g)=\rho(\g)T(g)\}$ is Zariski closed and contains $G(\R)$.

Let us now focus on the graph of the representation $\rho:\Gamma\to H$ as a subset $\Gr(\rho)$ of the group $ G\times H$. Since $\rho$ is an homomorphism, $\Gr(\rho)$ is a subgroup of $G\times H$, hence its Zariski closure $\ov{\Gr(\rho)}^Z$ is an algebraic subgroup. The image under the first projection $\pi_1$ of $\ov{\Gr(\rho)}^Z$ is a closed subgroup of $G$: indeed the image of a rational morphism (over an algebraically closed field) contains an open subset of its closure, since in our case $\pi_1$ is a group homomorphism, its image is an open  subgroup that is hence also closed. Moreover $\pi_1(\ov{\Gr(\rho)}^Z)$ contains $\G$ that is Zariski dense in $G$ by Borel density, hence equals $G$.

We now want to use the existence of the algebraic map $T$ and the fact that $\rho(\G)$ is Zariski dense in $H$ to show that $\ov{\Gr(\rho)}^Z$ is the graph of an homomorphism.
In fact it is enough to show that $\ov{\Gr(\rho)}^Z\cap (\{{\rm id}\}\times H)=({\rm id},{\rm id})$.
Let $({\rm id},f)$ be an element in $\ov{\Gr(\rho)}^Z\cap (\{{\rm id}\}\times H)$. Since $H$ is absolutely simple being an adjoint form of a simple Lie group, and $N=\bigcap_{h\in H}hQh^{-1}$ is a normal subgroup of $H$, it is enough to show that $f\in N$ or, equivalently, that $f$ fixes pointwise $H/Q$.  

But, since $T$ is a regular map, and the actions of $G$ on itself and of $H$ on $H/Q$ are algebraic, we get that the stabilizer of the map $T$ under the $G\times H$- action,
 $$\Stab_{G\times H}(T)=\{(g,h)|\;((g,h)\cdot T)(x)=h^{-1}T(gx)=T(x),\;\forall x\in G\},$$ is a Zariski closed subgroup of $G\times H$. Moreover $\Stab_{G\times H}(T)$  contains $\Gr(\rho)$ and hence also $\ov{\Gr(\rho)}^Z$. In particular $({\rm id},f)$ belongs to the stabilizer of $T$, hence the element $f$ of $H$ fixes the image of $T$ pointwise. Since the image of $T$ is $\rho(\G)$-invariant, $\rho(\G)$ is Zariski dense and the set of points in $H/Q$ that are fixed by $f$ is a closed subset, $f$ acts trivially on $H/Q$.
\end{proof}
We can now prove  Theorem \ref{thm:general}:

\begin{proof}[Proof of Theorem \ref{thm:general}]
Let $\rho:\G\to \SU(m,n)$ be a maximal representation and let $L$ be the Zariski closure of $\rho(\G)$ in $\SL(m+n,\C)$. Here, as above, $\SU(m,n)=H(\R)$ with respect to a suitable real structure on $H=\SL(m+n,\C)$. Since the representation $\rho$ is tight, we get, as a consequence of Theorem \ref{thm:tight}, that $L(\R)$ almost splits a product $L_{nc}\times L_{c}$ where $L_{nc}$ is a semisimple Hermitian Lie group tightly embedded in $\SU(m,n)$ and $K=L_c$ is a compact subgroup of $\SU(m,n)$. 

Let us consider $L_1,\ldots, L_k$ the simple factors of $L_{nc}$, namely $L_{nc}$, being semisimple, almost splits as the product $L_1\times\ldots \times L_k$ where $L_k$ are simple Hermitian Lie groups. The first observation is that none of the groups $L_i$ can be virtually isomorphic to $\SU(1,1)$. In that case the composition of the representation $\rho$ with the projection $L_{nc}\to L_i$ would be a maximal representation of a complex hyperbolic lattice with values in a group that is virtually isomorphic to $\SU(1,1)$ and this is ruled out by \cite{BICartan}: indeed Burger and Iozzi prove, as the last step in their proof of \cite[Theorem 2]{BICartan}, that there are no maximal representations of complex hyperbolic lattices in $\PU(1,1)$.

This implies that the inclusion $i:L_{nc}\to \SU(m,n)$ fulfills the hypotheses of Theorem \ref{thm:tightol}. In particular it  is enough to prove that each factor $L_i$ which is not of tube type is isomorphic to $\SU(1,p)$ and the composition of $\rho$ with the projection to $L_i$ is conjugate to the inclusion. 
Since, by Corollary \ref{cor:noZariskidense}, there is no Zariski dense representation of $\G$ in $\SU(m_i,n_i)$ if $1<m_i<n_i$, we get that $m_i=1$. Moreover, since the only Zariski dense tight representation of $\SU(1,p)$ in $\SU(1,q)$ is the identity map, we get that $n_i=p$ and the composition of $\rho$ with the projection to $L_i$ is conjugate to the inclusion. 
This concludes the proof.

\end{proof}

\begin{proof}[Proof of Corollary \ref{cor:general}]
 We know that the Zariski closure of the representation $\rho$ is contained in a subgroup of $\SU(m,n)$ isomorphic to $\SU(1,p)^t\times \SU(m-t,m-t)\times K$. The product $M=\SU(1,p)^t\times \SU(m-t,m-t)$ corresponds to a splitting $\C^{m,n}=V_1\oplus\ldots\oplus V_t\oplus W\oplus Z$ where the restriction of $h$ to $V_i$ is non-degenerate and has signature $(1,p)$ and the restriction of $h$ to $W$ is non-degenerate and has signature $(m-t,m-t)$. The subspace $W$ is left invariant by $M$ hence also by $K$ (since $K$ commutes with $M$ and all the invariant subspaces for $M$ have different signature).
 In particular the linear representation of $\G$ associated with $\rho$ leaves invariant a subspace on which $h$ has signature $(k,k)$ for some $k$ greater than 1 unless there are no factors 
of 
tube-type in the decomposition of $L$. This latter case corresponds to standard embeddings. 
\end{proof}

\begin{proof}[Proof of Corollary \ref{cor:local rigidity}]
 Let us denote by $\rho_0:\G\to \SU(m,n)$ the standard representation. Since by Lemma \ref{lem:Toledo constant} the generalized Toledo invariant is constant on components of the representation variety, we get that any other representation $\rho$ in the component of $\rho_0$ is maximal. By Theorem \ref{thm:general} this implies that $\ov{\rho(\G)}^Z$ almost splits as a product $K\times L_t\times \SU(1,p)$, and is contained in a subgroup of $\SU(m,n)$ of the form $\SU(1,p)^t\times \SU(m-t,m-t)\times K$. If the group $L_t$ is trivial then $\rho$ is a standard embedding, and is hence conjugate to $\rho_0$ up to a character in the compact centralizer of the image of $\rho_0$. In particular this would imply that $\rho_0$ is locally rigid.
 
 Let us then assume by contradiction that there are representations $\rho_i$  arbitrarily close to $\rho_0$ and with the property that the tube-type factor of the Zariski closure of $\rho_i(\G)$ is non trivial. Up to modifying the representations $\rho_i$ we can assume that the compact factor $K$ in the Zariski closure of $\rho_i$ is trivial.
 
 By Theorem \ref{thm:general} this implies that $\rho_i(\G)$ is contained in a subgroup of $\SU(m,n)$ isomorphic to $\SU(m,pm-1)$, moreover we can assume, up to conjugate the representations $\rho_i$ in $\SU(m,n)$, that the Zariski closure of $\rho_i$ is contained in the same subgroup $\SU(m,pm-1)$ for every $i$. Since the representations whose image is contained in the subgroup $\SU(m,pm-1)$ is a closed subspace of $\Hom(\G,G)/G$, we get that the image of $\rho_0$ is contained in $\SU(m,pm-1)$ and this is a contradiction, since the image of the diagonal embedding doesn't leave invariant any subspace on which the restriction of $h$ has signature $(m,pm-1)$.
\end{proof}

\end{document}